\documentclass[a4paper, reqno]{amsart}
\usepackage{amsfonts, amsmath, amssymb, amsthm, color}
\usepackage{geometry}
\newtheorem{thm}{Theorem}[section]
\newtheorem{lemma}[thm]{Lemma}
\newtheorem{prop}[thm]{Proposition}

\theoremstyle{definition}
\newtheorem{rmk}[thm]{Remark}

\newcommand{\po}{\partial\Omega}
\newcommand{\oo}{\overline{\Omega}}
\newcommand{\rn}{\mathbb{R}^n}
\newcommand{\intom}{\int_{\Omega}}
\newcommand{\intpom}{\int_{\partial\Omega}}
\newcommand{\intai}{\int_{A_i}}
\newcommand{\intal}{\int_{A_l}}

\newcommand{\ho}{H^1_0(\Omega)}

\newcommand{\udx}{U_{\delta, \xi}}
\newcommand{\pdxn}{\psi_{\delta, \xi}^0}

\newcommand{\md}{\mathbf{d}}
\newcommand{\mt}{\mathbf{t}}
\newcommand{\mh}{\mathcal{H}(\Omega)}
\newcommand{\mhr}{\mathcal{H}(\rn)}

\newcommand{\wje}{\widetilde{J}_{\ep}}
\newcommand{\pt}{\partial_t}
\newcommand{\pr}{\partial_r}
\newcommand{\pdl}{\partial_{d_l}}

\newcommand{\pxn}{\partial_{x_n}}

\newcommand{\sik}{\sum_{i=1}^k}
\newcommand{\slk}{\sum_{l=1}^k}
\newcommand{\sij}{\sum_{i,j}}
\renewcommand{\(}{\left(}
\renewcommand{\)}{\right)}

\newcommand{\ep}{\epsilon}
\newcommand{\dt}{_{\md, \mt}}
\newcommand{\pe}{^{\perp}}

\newcommand{\la}{\left\langle}
\newcommand{\ra}{\right\rangle}

\newcommand{\tp}{\tilde{p}}

\begin{document}
\title[Boundary towers of layers for supercritical problems]{Boundary towers of layers for some supercritical problems}
\author{Seunghyeok Kim}
\address[Seunghyeok Kim] {Department of Mathematics, Pohang University of Science and Technology, Pohang, Kyungbuk 790-784, Republic of Korea}
\email{shkim0401@gmail.com}

\author{Angela Pistoia}
\address[Angela Pistoia] {Dipartimento SBAI, Universit\`{a} di Roma ``La Sapienza", via Antonio Scarpa 16, 00161 Roma, Italy}
\email{pistoia@dmmm.uniroma1.it}

\begin{abstract}
We consider the supercritical problem%
\begin{equation*}
-\Delta u=\left\vert u\right\vert ^{p-1}u\text{ \ in }\mathcal D,\quad u=0\text{
\ on }\partial\mathcal D,
\end{equation*}
where $\mathcal D$ is a bounded smooth domain in $\mathbb{R}^{N}$ and $p+1$ is smaller
than the $\kappa-$th critical Sobolev exponent $2_{N,\kappa}^{\ast}:=\frac{N-\kappa+2}{N-\kappa-2}$ with
$1\leq \kappa\leq N-3.$ We show that in some suitable torus-like domains $\mathcal D$ there exists an
arbitrary large number of   sign-changing solutions with alternate positive and negative layers which
concentrate at different rates along a $\kappa$-dimensional submanifold  of
$\partial\mathcal D$ as $p$ approaches $2_{N,\kappa}^{\ast}$ from below.
\end{abstract}

\date{\today}
\subjclass[2010]{35J60, 35J20}
\keywords{Nonlinear elliptic boundary value problem; supercritical
exponents; existence sign changing solutions}
\thanks{The first author is partially supported by TJ Part Doctoral Fellowship funded by POSCO}
\maketitle
\numberwithin{equation}{section}

\section{Introduction}
This paper deals with the classical Lane-Emden-Fowler problem
\begin{equation}\label{pb1}
\Delta v+|v|^{p-1}v=0 \quad \text{in } \mathcal D, \qquad v = 0 \quad \text{on } \partial\mathcal D
\end{equation}
where $ \mathcal D$ is a bounded smooth domain in $\mathbb {R}^N,$ $N\ge3$ and $p>1.$
In particular, we are interested in exploring the role of the lower-dimensional Sobolev exponents $2^*_{N,\kappa}$ on the existence and multiplicity of solutions to problem \eqref{pb1}.
For any integer $\kappa$ between $0$ and $N-2$ let us set
\begin{equation}\label{2sk}
2^*_{N,\kappa}:={N-\kappa+2 \over N-\kappa-2}\quad\text{if } 0\le\kappa\le N-3\quad\text{and}\quad 2^*_{N,N-2}:=+\infty.
\end{equation}
If $ 0\le\kappa\le N-3, $  then $2^*_{N,\kappa}+1$ is nothing but the $\kappa-$th critical Sobolev exponent in dimension $N-\kappa.$

It is well known that in the subcritical regime, i.e. $p<2^*_{N ,0}$, the compactness of the Sobolev embedding ensures the existence of at least one positive
solution  and infinitely many sign-changing solutions to \eqref{pb1}.

In the critical case (i.e. $p=2^*_{N,0}$) or in the supercritical case (i.e. $p>2^*_{N,0}$) existence of solutions to problem \eqref{pb1} turns out to be a delicate issue. Indeed, if the domain $\mathcal{D} $ is star shaped Pohozaev's identity \cite{Po} implies that problem \eqref{pb1} has only the trivial solution.

In the critical case, if $\mathcal{D} $ has nontrivial reduced homology with $\mathbb{Z}_2-$coefficients, Bahri-Coron \cite{BC} proved that problem \eqref{pb1} has a positive solution in the critical case. Moreover, it was proved by Ge-Musso-Pistoia \cite{GMP} and Musso-Pistoia \cite{MP2} that if $\mathcal{D} $ has a small hole, problem
\eqref{pb1} has many sign changing solutions, whose number increases as the diameter of the hole decreases.

In the supercritical regime the existence of a nontrivial homology class in $\mathcal{D}$ does not guarantee the existence of a nontrivial solution to \eqref{pb1}. Passaseo in \cite{Pa1,Pa2} exhibited a domain in $\mathbb {R}^N$ homotopically equivalent to the $\kappa-$dimensional sphere in which problem \eqref{pb1}
with $p\ge 2^*_{N ,\kappa}$ has only the trivial solution.
Recently Clapp-Faya-Pistoia \cite{CFP} built  domains in $\mathbb {R}^N$ with a richer topology, namely the cup-length is $\kappa+1$,  in which problem \eqref{pb1} with $p>2^*_{N,\kappa}$ has only the trivial solution.
When  $p=2^*_{N,\kappa}$ the existence of infinitely many positive solutions to \eqref{pb1} was proved by Wei-Yan \cite{WY} for suitable torus-like domains $\mathcal{D}.$

\medskip
It is interesting to study problem \eqref{pb1} in the almost critical case, i.e. $p=2^*_{N,\kappa}\pm\ep,$ where $\ep$ is a small positive parameter.

The peculiarity of the almost critical case when $\kappa=0$ is that problem \eqref{pb1} has solutions which blow-up at one or more simple or multiple points in $\mathcal{D} $ as $\ep$ goes to zero. Indeed, if $p=2^*_{N,0}-\ep,$ positive  and sign-changing solutions  to \eqref{pb1} with different  simple blow-up points were built by Bahri-Li-Rey \cite{BLR} and Bartsch-Micheletti-Pistoia \cite{BMP}, respectively. Moreover, Pistoia-Weth \cite{PW} and Musso-Pistoia \cite{MP} proved that the number of sign-changing solutions  to \eqref{pb1}  with  a multiple blow-up point increases as $\ep$ goes to zero.
On the other hand, if $p=2^*_{N,0}+\ep,$  Ben Ayed-El Mehdi-Grossi-Rey \cite{BEGR} proved that problem \eqref{pb1} does not have any positive solutions with one positive blow-up point, while  Del Pino-Felmer-Musso \cite{DFM} and Pistoia-Rey \cite{PR} found solutions with two or more positive blow-up points provided the domain $\mathcal D$ has a hole.
Up to our knowledge there are no  results about existence of sign-changing solutions in this case. In particular, we quote Ben Ayed-Bouh \cite{BB} who  proved that problem \eqref{pb1} does not have any sign-changing solutions with one positive and one or two negative blow-up points.

\medskip
Having in mind what happens in the almost critical case  when $\kappa=0,$ we wonder if   the same phenomenon occurs for any $1\le\kappa\le N-2.$
More precisely, we ask if for some suitable domains $\mathcal D$ the problem \eqref{pb1} has solutions which blow-up at one or more simple or multiple
$\kappa-$dimensional manifolds in $\mathcal D$ as $p$ approaches the $\kappa-$th Sobolev exponent $2^*_{N,\kappa} $ from below.
A first result in this direction was obtained by Del Pino-Musso-Pacard \cite{DMP}.
If $\kappa=1$ and $p=2^*_{N,1}-\ep,$  they proved that for some domains $\mathcal D$ if $\ep$ is different from an explicit set of values,
problem \eqref{pb1} has a positive solution  which concentrates along a $1-$dimensional submanifold of the boundary of $\mathcal D $ when $\ep$ goes to zero.
Recently, it has been showed that if   $ \kappa\ge 2$ and $p$ approaches from below $2^*_{N,\kappa}$ it is possible to build torus-like  domains $\mathcal D$ in which problem \eqref{pb1} has
 positive solutions which concentrate at a $\kappa-$dimensional submanifold of   $\partial\mathcal D$. The construction was performed in the case $ 1\le \kappa \le N-3 ,$ $p =2^*_{N,\kappa} -\ep $ and $\ep$ goes to zero and in the case $\kappa=N-2$ and $p$ goes to $+\infty$ by Ackermann-Clapp-Pistoia \cite{ACP} and Kim-Pistoia \cite{KP},
  respectively.

As far as it concerns existence of sign-changing solutions, when $ 1\le \kappa \le N-3, $   $p =2^*_{N,\kappa} -\ep $ and $\ep$ is small enough or when  $\kappa=N-2$ and $p$ is large enough, Ackermann-Clapp-Pistoia \cite{ACP} and  Kim-Pistoia \cite{KP}, respectively,   constructed a sign-changing solution  with a positive and a negative layer which concentrate with the same rate along the same $\kappa-$dimensional submanifold of   the boundary of   suitable  torus-like domains $ \mathcal D,$
  as $\ep$ goes to zero. In particular, Kim-Pistoia \cite{KP} proved that when $\kappa=N-2$ the number of sign changing solutions to \eqref{pb1} increases as $p$ goes to $+\infty,$ provided $\mathcal D$ satisfies some symmetric assumptions. Their solutions have an arbitrary number of alternate positive and negative layers which concentrate with the same rate along the same  $(N-2)-$dimensional submanifold of $\partial\mathcal D$ as $p$ goes to $+\infty.$

 \medskip
  In this paper, we build domains $\mathcal D$ such that    the number of sign-changing solutions of
problem \eqref{pb1} when $1\le\kappa\le N-3 $ and $p=2^*_{N,\kappa} -\ep$ increases as $\ep$ goes to zero.
 In particular, for
each set of positive integers $\kappa_{1},\ldots,\kappa_{m}$ with $\kappa:=\kappa_{1}+\cdots
+\kappa_{m}\leq N-3$ we exhibit torus-like domains $\mathcal{D}$ for which the number of sign-changing solutions to problem
\eqref{pb1} with $p=2^*_{N,\kappa} -\ep$ increases as $\ep$ goes to zero.These solutions have an arbitrary large number
of alternate positive and  negative layer which concentrate with different rates along a
$\kappa$-dimensional submanifold $\Gamma_{0}$ of   $\partial\mathcal{D}$
which is diffeomorphic to the product of spheres $\mathbb{S}^{\kappa_{1}}%
\times\cdots\times\mathbb{S}^{\kappa_{m}}.$ This
follows from our main results, which we next state.

Fix $\kappa _{1},\ldots,\kappa _{m}\in\mathbb{N}$ with $\kappa :=\kappa _{1}+\cdots+\kappa _{m}\leq N-3$ and
a bounded smooth domain $\Omega$ in $\mathbb{R}^{N-\kappa }$ such that
\begin{equation}
\overline{\Omega}\subset\{\left(  x_{1},\ldots,x_{m},x^{\prime}\right)
\in\mathbb{R}^{m}\times\mathbb{R}^{N-\kappa -m}:x_{i}>0,\text{ }i=1,\ldots,m\}.
\label{omega}%
\end{equation}
Set
\begin{equation}
\mathcal{D}:=\{(y^{1},\ldots,y^{m},z)\in\mathbb{R}^{\kappa _{1}+1}\times\cdots
\times\mathbb{R}^{\kappa _{m}+1}\times\mathbb{R}^{N-\kappa -m}:\left(  \left\vert
y^{1}\right\vert ,\ldots,\left\vert y^{m}\right\vert ,z\right)  \in\Omega\}.
\label{D}%
\end{equation}
$\mathcal{D}$ is a bounded smooth domain in $\mathbb{R}^{N}$ which is
invariant under the action of the group $\Theta:=O(\kappa _{1}+1)\times\cdots\times
O(\kappa _{m}+1)$ on $\mathbb{R}^{N}$ given by%
\begin{equation*}
(g_{1},\ldots,g_{m})(y^{1},\ldots,y^{m},z):=(g_{1}y^{1},\ldots,g_{m}y^{m},z).
\end{equation*}
for every $g_{i}\in O(\kappa _{i}+1),$ $y^{i}\in\mathbb{R}^{\kappa _{i}+1},$
$z\in\mathbb{R}^{N-\kappa -m}.$ Here, as usual, $O(d)$ denotes the group of linear
isometries of $\mathbb{R}^{d}.$ For $p ={2}_{N,\kappa }^{\ast} -\ep$ we shall
look for $\Theta$-invariant solutions to problem \eqref{pb1}, i.e.
solutions $v$ of the form%
\begin{equation}
v(y^{1},\ldots,y^{m},z)=u(\left\vert y^{1}\right\vert ,\ldots,\left\vert
y^{m}\right\vert ,z). \label{inv}%
\end{equation}
A simple calculation shows that $v$ solves problem \eqref{pb1} if and only
if $u$ solves
\begin{equation*}
-\Delta u-\sum_{i=1}^{m}\frac{\kappa_{i}}{x_{i}}\frac{\partial u}{\partial x_{i}%
}=|u|^{p-1}u\quad\text{in}\ \Omega,\qquad u=0\quad\text{on}\ \partial\Omega.
\end{equation*}
This problem can be rewritten as%
\begin{equation*}
-\text{div}(a(x)\nabla u)=a(x)|u|^{p-1}u\quad\text{in}\ \Omega,\qquad
u=0\quad\text{on}\ \partial\Omega,
\end{equation*}
where $a(x_{1},\ldots,x_{N-\kappa}):=x_{1}^{\kappa_{1}}\cdots x_{m}^{\kappa_{m}}.$ Note that
${2}_{N,\kappa}^{\ast}$ is the critical exponent in dimension $n:=N-\kappa$ which is the
dimension of $\Omega.$

Thus, we are lead to study the more general almost critical problem
\begin{equation}\label{main}
-\text{div}(a(x) \nabla u) = a(x) |u|^{{4 \over n-2}-\ep} u \quad \text{in } \Omega, \qquad u = 0 \quad \text{on } \po
\end{equation}
where $\Omega$ is a smooth bounded domain in $\rn,$  $n \ge 3,$  $\ep > 0$ is a    small parameter   and $a \in C^2(\overline\Omega)$  is strictly positive in $\overline\Omega.$

This is a subcritical problem, so standard variational methods yield one
positive and infinitely many sign changing solutions to problem \eqref{main} for
every $\epsilon\in(0,{\frac{4}{n-2})}$. Our
goal is to construct solutions $u_{\epsilon}$ with an arbitrary large number of alternate positive and negative
bubbles which accumulate with different rates at the same point  $\xi_{0}$ of
$\partial\Omega$ as $\epsilon\rightarrow0.$ They correspond, via (\ref{inv}),
to $\Theta$-invariant solutions $v_{\epsilon}$ of problem \eqref{pb1} with
positive and negative layers which accumulate with different rates along the $\kappa$-dimensional
submanifold
\begin{equation*}
\Gamma _0:=\{(y^{1},\ldots,y^{m},z)\in\mathbb{R}^{\kappa_{1}+1}\times\cdots
\times\mathbb{R}^{\kappa_{m}+1}\times\mathbb{R}^{N-\kappa-m}:\left(  \left\vert
y^{1}\right\vert ,\ldots,\left\vert y^{m}\right\vert ,z\right)  =\xi_{0}\}
\end{equation*}
of the boundary of $\mathcal{D}$ as $\epsilon\rightarrow0.$ Note that
$M_{0}$ is diffeomorphic to $\mathbb{S}^{\kappa_{1}}\times\cdots\times
\mathbb{S}^{\kappa_{m}}$ where $\mathbb{S}^{d}$ is the unit sphere in
$\mathbb{R}^{d+1}.$

We will assume the following conditions.
 \begin{enumerate}
\item[(a1)] There are constants $a_1$ and $a_2$ such that
\[0 < a_1 \le a(x) \le a_2 < +\infty \quad \text{for all } x \in \oo.\]

\item[(a2)] The restriction of $a$ to $\po$ has a critical point $\xi_0 \in \po$ and
\[\partial_{\nu} a(\xi_0) := (\nabla a(\xi_0), \nu(\xi_0)) > 0 \]
where $\nu := \nu(\xi_0)$ is the inward unit normal vector to $\po$ at $\xi_0$.

\item[(a3)] The domain $\Omega$ and the function $a$ are symmetric with respect to the direction given
by $\nu(\xi_0)$, i.e.,
\begin{multline*}
\hspace{30pt}(x, \nu) \nu + (x, \tau_1) \tau_1 + \cdots + (x, \tau_i) \tau_i + \cdots (x, \tau_{n-1}) \tau_{n-1} \in \Omega\\
\Leftrightarrow (x, \nu) \nu + (x, \tau_1) \tau_1 + \cdots - (x, \tau_i) \tau_i + \cdots (x, \tau_{n-1}) \tau_{n-1} \in \Omega
\end{multline*}
and
\begin{multline*}
\hspace{30pt} a\left((x, \nu) \nu + (x, \tau_1) \tau_1 + \cdots + (x, \tau_i) \tau_i + \cdots (x, \tau_{n-1}) \tau_{n-1} \right)\\
= a\left((x, \nu) \nu + (x, \tau_1) \tau_1 + \cdots - (x, \tau_i) \tau_i + \cdots (x, \tau_{n-1}) \tau_{n-1}\right)
\end{multline*}
for $i= 1, \cdots, n-1$. Here $(\cdot, \cdot)$ is the standard inner product in $\rn$ and $\{\tau_1, \cdots, \tau_{n-1}\}$ is an orthonormal basis of the tangent space $T_{\xi_0}\partial \Omega$.
\end{enumerate}

For each $\delta>0,\ \xi\in\mathbb{R}^{n},$\ we consider the standard bubble
\begin{equation*}
U_{\delta,\xi}(x):=[n(n-2)]^{\frac{n-2}{4}}{\frac{\delta^{\frac{n-2}{2}}%
}{\left(  \delta^{2}+|x-\xi|^{2}\right)  ^{\frac{n-2}{2}}}.}%
\end{equation*}

\medskip

We will prove the following result.
\begin{thm}\label{thm_main} Suppose that $(a1)-(a3)$ hold true for $a$ and $\Omega$. Also, assume that $n\ge4.$
For any integer $k$, there exists $\ep_k > 0$ such that for each $0 < \ep < \ep_k$ problem \eqref{main}
has a sign changing solution $u_{\ep}$ which satisfies
\[ u_{\ep} = \sik (-1)^{i+1} U_{\delta_i(\ep), \xi_i(\ep)} + o(1)\quad\text{in}\quad H^1_0(\Omega)\]
where  \[ \ep^{-{n-1+2(i-1) \over n-2}} \delta_i(\ep) \to d_i > 0, \quad \xi_i(\ep) \to \xi_0 \in \po \quad \text{as } \ep \to 0\]
for $i= 1, \cdots, k$.
\end{thm}

\medskip \noindent

The solutions we found resemble the towers  of bubbles with alternating sign which concentrates at a point on the boundary of $\Omega.$
This kind of solutions is typical of almost critical problems (see \cite{DDM,DMPi,GJP,PW,MP}).

 The symmetry of the domain $\Omega$ as stated in $(a2)$
allows to simplify considerably the computations. We believe that the result is true if
we only require that $\xi_0$ is a non degenerate critical point of the restriction of $a$ to the $\partial\Omega.$
Moreover, the restriction on the dimension $n\ge4$ is due to technical reasons as it is explained in   Remark \ref{rmk_pre}.
We also believe that it can be removed but it seems to be necessary to overcome some technical difficulties.
\medskip

Now, we come back to problem \eqref{pb1}.
In the following   theorem  we assume that we are given $\kappa_{1},\ldots,\kappa_{m}%
\in\mathbb{N}$ with $\kappa :=\kappa_{1}+\cdots+\kappa_{m}\leq N-3$ and a bounded smooth
domain $\Omega$ in $\mathbb{R}^{N-\kappa}$ which satisfies (\ref{omega}). We set
$a(x_{1},\ldots,x_{N-\kappa}):=x_{1}^{ \kappa_{1}}\cdots x_{m}^{\kappa_{m}},$ $\mathcal{D}$
as in (\ref{D}), $p =2_{N,\kappa}^{\ast}-\epsilon,$ $\Theta:=O(\kappa_{1}+1)\times
\cdots\times O(\kappa_{m}+1)$ and
\begin{equation*}
\widetilde{U}_{\delta,\xi}(y^{1},\ldots,y^{m},z):=U_{\delta,\xi}(\left\vert
y^{1}\right\vert ,\ldots,\left\vert y^{m}\right\vert ,z)
\end{equation*}
for $\delta>0,\ \xi\in\mathbb{R}^{N-\kappa}.$

\begin{thm}
\label{main1} Assume $n = N-\kappa\ge 4.$ Then
for any integer $k$ there exists $\epsilon_{k}>0$ such that  for  any $\epsilon\in(0,\epsilon_{0}),$ problem
\eqref{pb1} has a $\Theta$-invariant solution $v_{\epsilon}$ which
satisfies
\begin{equation*}
v_{\epsilon}(x)=\sum\limits_{i=1}^{k}(-1)^{i+1}\widetilde
{U}_{\delta_ i(\epsilon),\xi_ i(\epsilon)}(x)+o(1)\qquad\text{in
}H^1_0(\mathcal{D}),
\end{equation*}
with
\begin{equation*}
\epsilon^{-{n-1+2(i-1) \over n-2}}\delta_i(\epsilon)\rightarrow d_{i}%
>0\qquad\text{and}\qquad\xi_i(\epsilon)\rightarrow\xi_{0}\in
\partial\Omega,
\end{equation*}
for each $i=1,\dots, k$ as $\epsilon\rightarrow0.$
\end{thm}
The solutions we found resemble    the towers  of layers with alternating sign which concentrate  at a $\kappa-$dimensional submanifold of the boundary of $\mathcal D.$
This result extends the one obtained  by Pistoia-Weth \cite{PW} and Musso-Pistoia \cite{MP} when $\kappa=0$ to higher $\kappa$'s.
 Moreover, we stress the fact  that the  profile of our solutions    is different from the one   found by Ackermann-Clapp-Pistoia \cite{ACP} and  Kim-Pistoia \cite{KP}.
 Indeed, their solutions look like a cluster of layers (i.e. all the layers concentrate at the same speed), while our solution look like a tower of layers
(i.e. one layer   concentrates faster than the previous one).

\medskip
It is interesting to prove that this kind of solutions also exists in the setting of \cite{DMP}. Indeed, we conjecture that if $\Gamma$ is a nondegenerate
geodesic of the boundary of $\mathcal D$ with inner normal curvature it is possible to build towers of sign-changing solutions whose $1-$dimensional layers
concentrate at $\Gamma$ as $p$ approaches the first Sobolev critical exponent  $2^*_{N,1}$ from below (up to a subsequence of values).

\medskip
By the previous discussion Theorems \ref{main1}  follows
immediately from Theorems \ref{thm_main}. The proof of Theorem
\ref{main} relies on a very well known Ljapunov-Schmidt
reduction. We omit many details on the finite dimensional reduction because
they can be found, up to some minor modifications, in the literature. We only
compute what cannot be deduced from known results. In Section \ref{var-set} we
write the approximate solution, we sketch the proof of the Ljapunov-Schmidt
procedure and we prove Theorem  \ref{main1}. In Section \ref{sec_error} we compute
the rate of the error term, while in Section \ref{sec_expansion_1} and in Section \ref{sec_expansion_2}
we give the $C^0-$estimate and the $C^1-$estimate of  the reduced
energy, respectively. In Appendix A we give some important estimates which are not available in the literature.

\medskip \noindent
\textbf{Notations.}

\noindent - For the sake of convenience, we assume that $\xi_0 = 0 \in \rn$, $\tau_i = e_i$ for $i = 1, \cdots, n-1$ and $\nu = e_n$ where $\{e_1, \cdots, e_n\}$ denotes the standard basis in $\rn$. Thus assumption (a3) reads as $\Omega$ is symmetric with respect to the $x_n$-axis and $a(x_1, \cdots, x_i, \cdots, x_n) = a(x_1, \cdots, -x_i, \cdots, x_n)$ for $i = 1, \cdots, n-1$.

\noindent -   $D^{1,2}(\rn)$ is the space of measurable and
weakly differentiable functions the $L^2$-norms of whose gradient are finite.

\noindent - $\mathcal{D}(\Omega)$ is the space of smooth functions whose supports are compactly contained
in $\Omega$ and $\ho$ is the completion of $\mathcal{D}(\Omega)$  with respect to the norm
$\|u\| = \la u,u \ra^{1 \over 2} = \(\intom a |\nabla u|^2\)^{1 \over 2}$.
By virtue of (a1), this norm is equivalent to the usual one.

\noindent - $\mh$ is a subspace of $\ho$ defined by
\[\mh = \{u \in \ho  : u(x_1, \cdots, x_i, \cdots, x_n) = u(x_1, \cdots, -x_i, \cdots, x_n) \text{ for each } i= 1, \cdots, n-1\}.\]
Also, $\mhr$ is a subspace of $D^{1,2}(\rn)$ defined similarly.

\noindent - For any $x \in \rn$ and $r > 0$, $B(x, r)$ is the open ball in $\rn$ of radius $r$ centered at $x$.

\noindent - $|B^n| = \pi^{n/2}/\ \Gamma(n/2+1)$ and $|S^{n-1}| = (2 \pi^{n/2})/\ \Gamma(n/2)$ denotes the Lebesgue measure of the $n$-dimensional unit ball and $(n-1)$-dimensional unit sphere, respectively.

\noindent - We will use big $O$ and small $o$ notations to describe the limit behavior of a certain
quantity as $\ep \to 0$.

\noindent - $C > 0$ is a generic constant that may vary from line to line.

\section{Preliminaries and scheme of the proof of Theorem \ref{thm_main}} \label{var-set}
\subsection{An approximation for the solution}
Set $\alpha_n = [n(n-2)]^{n-2 \over 4}$ and let
\begin{equation}\label{instanton}
\udx(x) := \alpha_n {\delta^{n-2 \over 2} \over {(\delta^2 + |x -\xi|^2)}^{n-2 \over 2}} \quad \text{ for } \delta > 0, \ \xi = (\xi_1, \cdots, \xi_{n-1}, 0) \in \rn,
\end{equation}
which are positive solutions to the problem
\begin{equation}\label{limit_eq}
- \Delta u = u^{n+2 \over n-2} \quad \text{in } \rn, \quad u \in \mhr.
\end{equation}
Define also
\begin{equation}\label{pdxn}
\pdxn(x) := {\partial \udx \over \partial \delta} = \alpha_n \({n-2 \over 2}\) \delta^{n-4 \over 2} {|x - \xi|^2 -\delta^2 \over (\delta^2 + |x - \xi|^2)^{n \over 2}}
\end{equation}
and
\begin{equation}\label{pdxj}
\psi_{\delta, \xi}^i(x) := {\partial \udx \over \partial \xi_i} = \alpha_n (n-2) \delta^{n-2 \over 2} {(x - \xi)_i \over (\delta^2 + |x - \xi|^2)^{n \over 2}},\quad
 i=1,\dots,n,
\end{equation}
where $(x-\xi)_i$ is the $i$-th coordinate of $x - \xi \in \rn$.
Recall that the space spanned by $\pdxn,\psi_{\delta, \xi}^1,\dots,\psi_{\delta, \xi}^n$ is the set of bounded solutions to the linearized problem of \eqref{limit_eq} at $\udx$
\begin{equation}\label{eq_of_psi}
-\Delta \psi = \({n+2 \over n-2}\) \cdot \udx^{4 \over n-2} \psi \quad \text{in } \rn, \quad \psi \in D^{1,2}(\rn).
\end{equation}
In particular, the set of bounded solutions to the linear equation \eqref{eq_of_psi} in the space $\mhr$ is generated by the only two functions $\psi_{\delta, \xi}^0$
and $\psi_{\delta, \xi}^n.$

\medskip
Let $PW$ be the projection of the function $W \in D^{1,2}(\rn)$ onto $\ho$, that is,
\begin{equation}\label{proj}
\Delta PW = \Delta W \quad \text{ in } \Omega, \qquad PW = 0 \quad \text{ on } \po,
\end{equation}
and $k$ a fixed integer. (See Appendix \ref{appendix_W_PW} for estimation of $P\udx$ in terms of $\udx$.)  We look for a solution to problem \eqref{main} of the form
\[ u = \sik (-1)^{i+1}PU_{\delta_i, \xi_i} + \phi \in \mh\]
where the concentration parameters satisfy
\begin{equation}\label{conc_para}
\delta_i = \ep^{n-1+2(i-1) \over n-2} d_i \quad \text{with} \quad d_i > 0,
\end{equation}
the concentration points satisfy
\begin{equation}\label{conc_pt}
\xi_i = (\xi_0 + \ep t \nu(\xi_0)) + \delta_i s_i \nu(\xi_0) \quad \text{ with } t > 0 \text{ and } s_i \in \mathbb{R}, \ s_k = 0
\end{equation}
and $\|\phi\|$ is sufficiently small.

For simplicity we write $\md := (d_1, \cdots, d_k) \in (0, +\infty)^k$, $\mt := (t, s_1, \cdots, s_{k-1}) \in (0, +\infty) \times \mathbb{R}^{k-1} $,
$U_i =  U_{\delta_i, \xi_i}$ and
\begin{equation}\label{def_V}
V^{\ep}\dt = V\dt = \sik (-1)^{i+1} PU_i \in \mh.
\end{equation}
Also, we define the admissible set $\Lambda$ by
\[\Lambda = \left\{(\md, \mt): \md \in (0, +\infty)^k, \ \mt \in (0, +\infty) \times \mathbb{R}^{k-1} \right\}. \]

\subsection{Scheme of the proof of Theorem \ref{thm_main}}
First, we rewrite problem \eqref{main}. Let $i^*: L^{2n \over n+2}(\Omega) \to \ho$ be the
adjoint operator to the embedding $i : \ho \hookrightarrow L^{2n \over n-2}(\Omega)$,
i.e., $i^*(v) = u$ if and only if $\la u, \phi \ra = \intom av\phi$ for all $\phi \in \mathcal{D}(\Omega)$,
or $-\text{div}(a(x) \nabla u) = av$ in $\Omega$ and $u = 0$ on $\partial \Omega$.
Therefore \eqref{main} is equivalent to
\begin{equation}\label{main2}
u = i^*\(|u|^{p-1-\ep} u\), \quad u \in \ho \quad \text{where } p := {n+2 \over n-2}.
\end{equation}

For the sake of simplicity, we write $\psi_i^j = \psi_{\delta_i, \xi_i}^j$ with $\delta_i$ and $\xi_i$ defined in \eqref{conc_para}
and \eqref{conc_pt}. We introduce the spaces
\[K\dt = \text{span}\{P\psi_i^j : i = 1, \cdots, k, \ j = 0, n \}, \]
\begin{equation}\label{K_perp}
K\dt\pe = \left\{\phi \in \mh: \la \phi, P\psi_i^j \ra = 0 \text{ for } i = 1, \cdots, k, \ j = 0, n \right\},
\end{equation}
and the projection operators $\Pi\dt: \mh \to K\dt$ and $\Pi\dt\pe = Id_{\mh} - \Pi\dt: \mh \to K\dt\pe$.

As usual, we will solve problem \eqref{main2} by finding parameters $(\md, \mt) \in \Lambda$ and
a function $\phi \in K\dt\pe$ such that
\begin{equation}\label{es1}
\Pi\dt\pe\(V\dt + \phi - i^*\(|V\dt + \phi|^{p-1-\ep} (V\dt + \phi)\)\) = 0
\end{equation}
and
\begin{equation}\label{es2}
\Pi\dt\(V\dt + \phi - i^*\(|V\dt + \phi|^{p-1-\ep} (V\dt + \phi)\)\) = 0.
\end{equation}

\medskip
The first step is to solve equation \eqref{es1}. More precisely,  if  $\ep  $ is small enough for any fixed $(\md, \mt) \in \Lambda$, we will find a function
$\phi \in K\dt\pe$ such that \eqref{es1} holds.

First of all we define the linear operator $L\dt : K\dt\pe \to K\dt\pe$ by
\begin{equation}\label{def_L}
L\dt\phi =  \phi - \(p- \ep\) \cdot \Pi\dt\pe i^*\(|V\dt|^{p-1-\ep}\phi\).
\end{equation}

Arguing as in \cite[Lemma 3.1]{MP} and  using Lemma \ref{lemma_pre_5} and Lemma \ref{lemma_pre_4},  we   prove that it is invertible.

\begin{prop}\label{prop_linear}
For any compact subset $\Lambda_0$ of $\Lambda$, there exist $\ep_0 > 0$ and $c > 0$ such that for each
$\ep \in (0, \ep_0)$ and $(\md, \mt) \in \Lambda_0$ the operator $L\dt$ satisfies
\[\|L\dt \phi\| \ge c \|\phi\| \quad \text{for all } \phi \in K\dt\pe.\]
\end{prop}

Secondly, in Section \ref{sec_error} we estimate the error term
$$R\dt := \Pi\dt\pe \(i^*\(|V\dt|^{p-1-\ep} V\dt\) - V\dt\).$$
\begin{lemma}\label{lemma_R}
It holds true that
\[\|R\dt\| = O\(\ep^{{1 \over 2} \cdot {n+6 \over n+2}}\) = o\(\sqrt{\ep}\).\]
\end{lemma}

 Finally, we use a standard contraction mapping argument (see  \cite[Section. 5]{MP}) to solve equation \eqref{es1}.

\begin{prop}\label{prop_cont}
For any compact set $\Lambda_0$ of $\Lambda$, there is $\ep_0 > 0$ such that for each
$\ep \in (0, \ep_0)$ and $(\md, \mt) \in \Lambda_0$, a unique $\phi\dt^{\ep} \in K\dt\pe$ exists such that
\[\Pi\dt\pe\(V\dt + \phi\dt^{\ep} - i^*\(|V\dt + \phi\dt^{\ep}|^{p-1-\ep} (V\dt + \phi\dt^{\ep})\)\) = 0\]
and
\begin{equation}\label{est_phi}
\|\phi\dt^{\ep}\| = o\(\sqrt{\ep}\).
\end{equation}
\end{prop}

\medskip
The second step is to solve equation \eqref{es2}. More precisely, for $\ep$ small enough we will find $(\md, \mt)$ such that equation \eqref{es2} is satisfied.

Let us introduce the energy functional $J_{\ep}: \mh \to \mathbb{R}$ defined as
\[J_{\ep}(u) = {1 \over 2} \intom a(x)|\nabla u|^2 dx - {1 \over p + 1 - \ep} \intom a(x) |u|^{p+1-\ep}dx,\]
whose critical points are solutions to problem \eqref{main}  and let us define the reduced energy
functional $\wje: \Lambda \to \mathbb{R}$ by
\begin{equation}\label{red_energy}
\wje(\md, \mt) = J_{\ep}(V\dt + \phi\dt^{\ep}).
\end{equation}

First of all,  arguing as \cite[Proposition 2.2]{MP} and using Lemma \ref{lemma_pre_5} and Lemma \ref{lemma_pre_6}, we get
\begin{prop}\label{prop_red}
The function $V\dt + \phi\dt^{\ep}$ is a critical point of the functional $J_{\ep}$ if
the point $(\md, \mt)$ is a critical point of the function $\wje$.
\end{prop}

Thus, the problem is reduced to search for critical points of $\wje,$ whose asymptotic expansion is needed. The $C^0$ and $C^1$ estimates are carried out in Section \ref{sec_expansion_1}  and Section \ref{sec_expansion_2}, respectively, and they read as follows.

\begin{prop}\label{prop_energy_est}
It holds true that
\begin{equation}\label{energy_exp}
\wje(\md, \mt) = a(\xi_0)[c_1+c_2\ep-c_3\ep \log \ep] + \ep\Phi(\md, \mt)  + o(\ep),
\end{equation}
$C^1$-uniformly on compact sets of $\Lambda$.
Here, the function $\Phi: \Lambda \to \mathbb{R}$ is defined by
\begin{equation}\label{phi}
\Phi(\md, \mt) := \partial_{\nu} a(\xi_0) c_4 t + a(\xi_0) \left[ c_5 \({d_1 \over 2t}\)^{n-2} + c_6 \sum_{i=1}^{k-1} \({d_{i+1} \over d_i}\)^{n-2 \over 2} {1\over(1+s_i^2)^{n-2\over2}} \right] - a(\xi_0) c_7 \sum_{i=1}^{k} \log d_i
\end{equation}
where $c_i$'s are all positive constants.
\end{prop}

Finally, we can prove  Theorem \ref{thm_main} by
showing that $\wje$ has a critical point in $\Lambda$.
\begin{proof}[Proof of Theorem \ref{thm_main}]
The fact that $\partial_{\nu} a(\xi_0)$ is positive (see assumption (a2)) ensures that the function $\Phi$ defined in \eqref{phi} has a non-degenerate critical point of min-max type
(a minimum in $t$ and $d_i$'s and a maximum in $s_i$'s) which is stable under $C^1$-perturbations  (see Page 7 in \cite{MP}).
Therefore, by Proposition \ref{prop_energy_est}, we deduce that if $\ep  $
is small enough the function $\wje$ has a critical point. The claim follows by
Proposition \ref{prop_red}.
\end{proof}

\section{Estimate of the error term $R\dt$}\label{sec_error}
This section is devoted to prove Lemma \ref{lemma_R}.  For sake of brevity, we drop the subscript $\md, \mt.$

\medskip
  Using the definition of $V$ in \eqref{def_V}, we decompose first
\begin{equation}\label{cont_1}
\begin{aligned}
R &:= \Pi\pe \(i^*\(|V|^{p-1-\ep} V\) - V\) = \Pi\pe \bigg(i^*\bigg(|V|^{p-1-\ep} V - \sik (-1)^{i+1} U_i^p + \sik (-1)^{i+1} \nabla \log a \cdot \nabla PU_i\bigg)\bigg)\\
&=\Pi\pe \(i^*\(|V|^{p-1-\ep} V  - |V|^{p-1} V \)\) + \Pi\pe \bigg(i^*\bigg(|V|^{p-1} V - \sik(-1)^{i+1} PU_i^p \bigg)\bigg) \\
&\ + \sik(-1)^{i+1} \Pi\pe \(i^*\(PU_i^p - U_i^p\)\) + \sik (-1)^{i+1} \Pi\pe \(i^*\(\nabla \log a \cdot \nabla PU_i\)\) =: R_1 + R_2 + \sik R_3^i + \sik R_4^i.
\end{aligned}
\end{equation}

\medskip \noindent
\textit{Estimate of $R_1$.} Set $\tilde{p} := {2n \over n+2}$. By the boundedness of $i^*: L^{2n \over n+2}(\Omega) \to \ho$, the mean value theorem and
\begin{equation}\label{cont_2}
|u|^{q}|\log|u|| = O\(|u|^{q+\sigma} + |u|^{q-\sigma}\) \quad \text{for any } q > 1 \text{ and small } \sigma >0,
\end{equation}
it holds
\begin{align*}
\|R_1\|^{\tp} &\le \left\|i^*\(\(|V|^{p-1-\ep} - |V|^{p-1}\) V\)\right\|^{\tp}
\le C \left\|\(|V|^{p-1-\ep} - |V|^p-1\) V\right\|^{\tp}_{L^{\tp}(\Omega)}\\
&= C\ep^{\tp} \intom |\log|V||^{\tp}\cdot \sup_{\theta \in [0,1]}|V|^{(p - \theta \ep)\tp} \le C \ep^{\tp} \intom \(|V|^{p\tp - \sigma'} + |V|^{p\tp + \sigma'}\)\\
&\le C\ep^{\tp} \sik \intom \(U_i^{{2n \over n-2} - \sigma'} + U_i^{{2n \over n-2} + \sigma'}\) = O\(\ep^{\tp - \sigma''}\)
\end{align*}
where $\sigma'$ and $\sigma'' > 0$ are constants small enough. Hence
\begin{equation}\label{cont_3}
\|R_1\| = O\(\ep^{1 - \sigma}\) \quad \text{for any small } \sigma > 0.
\end{equation}

\medskip \noindent
\textit{Estimate of $R_2$.} Let $f(s) := |s|^{p-1}s$ for $s \in \mathbb{R}$ and choose $\rho > 0$ sufficiently small so that $\overline{B(\xi_k, \rho \ep)} \subset \Omega$. Following the approach introduced in \cite{MP}, we divide the domain $\Omega$ into $k+1$ mutually disjoint subsets, namely,
\[\Omega = \(\bigcup_{l=1}^k A_l\) \cup \(\Omega \setminus B(\xi_k, \rho \ep)\)
\]
where $A_l$'s are annuli defined as
\begin{equation}\label{A_l}
A_l = B\(\xi_k, \sqrt{\delta_{l-1}\delta_l}\) \setminus B\(\xi_k, \sqrt{\delta_l\delta_{l+1}}\) \quad \text{with } \delta_0 = {(\ep \rho)^2 \over \delta_1}, \ \delta_{k+1} = 0.
\end{equation}
Then by the mean value theorem,
\begin{align*}
\|R_2\|^{\tp} &\le C \bigg\||V|^{p-1} V - \sik(-1)^{i+1} PU_i^p\bigg\|^{\tp}_{L^{\tp}(\Omega)}
= C \sum_{l=1}^k \intal \bigg||V|^{p-1} V - \sik(-1)^{i+1} PU_i^p\bigg|^{\tp} + O\(\ep^{n \over n-2}\)\\
&= C \sum_{l=1}^k \intal \bigg|f\bigg((-1)^{l+1}PU_l + \sum_{i \ne l} (-1)^{i+1} PU_i\bigg) - f\bigg((-1)^{l+1}PU_l \bigg) \bigg|^{\tp} + O\(\ep^{n \over n-2}\)\\
&= O\(\sum_{l=1}^{k-1} \intal U_l^{(p-1)\tp} U_{l+1}^{\tp}\) + O\(\sum_{l=2}^k \intal U_l^{(p-1)\tp} U_{l-1}^{\tp}\) + O\(\ep^{n \over n-2}\).
\end{align*}
By \eqref{s23_4} and \eqref{s3_1.5} we deduce
\begin{align*}
\intal U_l^{(p-1)\tp} U_{l+1}^{\tp} &\le \left\|U_l^{8n \over n^2-4}U_{l+1}^{8n \over (n+2)^2}\right\|_{L^{(n+2)^2 \over 8n}(\Omega)}\left\|U_{l+1}^{2n(n-2) \over (n+2)^2}\right\|_{L^{(n+2)^2 \over (n-2)^2}(\Omega)}\\
&= \(\intal U_l^pU_{l+1}\)^{8n \over (n+2)^2} \cdot \(\intal U_{l+1}^{p+1} \)^{(n-2)^2 \over (n+2)^2}\\
&= O\(\ep^{8n \over (n+2)^2}\) \cdot O\(\ep^{n(n-2) \over (n+2)^2}\) = O\(\ep^{n(n+6) \over (n+2)^2}\)
\end{align*}
for $l = 1, \cdots, k-1$, and similarly
\[\intal U_l^{(p-1)\tp} U_{l-1} = O\(\ep^{n(n+6) \over (n+2)^2}\)\]
for $l = 2, \cdots, l$. Therefore we obtain
\begin{equation}\label{cont_4}
\|R_2\| = O\(\ep^{{1 \over 2} \cdot {n+6 \over n+2}}\) + O\(\ep^{{1 \over 2} \cdot {n+2 \over n-2}}\).
\end{equation}

\medskip \noindent
\textit{Estimate of $R_3$.} By the mean value theorem again,
\[\left \|R_3^i \right \|^{\tp} \le C \left \|PU_i^p - U_i^p \right \|^{\tp}_{L^{\tp}(\Omega)}
\le C \intom \(U_i^{(p-1)\tp}|PU_i - U_i|^{\tp} + |PU_i - U_i|^{p+1}\)\]
Arguing as in the proof of Lemma \ref{lemma_pre_3.5}, we get
\begin{equation}\label{cont_4.5}
\intom |PU_i - U_i|^{p+1} = O\(\ep^{n \over n-2}\),
\end{equation}
and
\begin{align*}
&~\intom U_i^{(p-1)\tp}|PU_i - U_i|^{\tp}\\
&\le \int_{B(\xi_i, \rho \ep)} U_i^{(p-1)\tp}|PU_i - U_i|^{\tp} + \(\int_{\Omega \setminus B(\xi_k, \rho \ep)} U_i^{p+1}\)^{4 \over n+2}\(\intom |PU_i - U_i|^{p+1}\)^{n-2 \over n+2}\\
&\le \(\int_{B(\xi_i, \rho \ep)} U_i^p|PU_i - U_i|\)^{8n \over (n+2)^2} \cdot \(\intom |PU_i - U_i|^{p+1} \)^{(n-2)^2 \over (n+2)^2} + O\(\ep^{4n \over n^2-4}\)\cdot O\(\ep^{n \over n+2}\)\\
&= O\(\ep^{n(n+6) \over (n+2)^2}\) + O\(\ep^{n(n+2) \over n^2-4}\)
\end{align*}
(see \cite[Lemma C.2 (64)]{ACP} for the estimate of the term $\int_{B(\xi_i, \rho \ep)} U_i^p|PU_i - U_i|$). Thus
\begin{equation}\label{cont_5}
\|R_3\| = O\(\ep^{{1 \over 2} \cdot {n+6 \over n+2}}\) + O\(\ep^{{1 \over 2} \cdot {n+2 \over n-2}}\).
\end{equation}

\medskip \noindent
\textit{Estimate of $R_4$.} Lemma \ref{lemma_pre_7} yields
\begin{equation}\label{cont_6}
\|R_4\| \le C \|\nabla PU_i\|_{L^{\tp}(\Omega)} = O(\ep)
\end{equation}

\medskip
In conclusion, from \eqref{cont_1}, \eqref{cont_3}, \eqref{cont_4}, \eqref{cont_5} and \eqref{cont_6}, we obtain
\[\|R\| = O\(\ep^{1 - \sigma}\) + O\(\ep^{{1 \over 2} \cdot {n+6 \over n+2}}\) + O\(\ep^{{1 \over 2} \cdot {n+2 \over n-2}}\) + O(\ep) = O\(\ep^{{1 \over 2} \cdot {n+6 \over n+2}}\).\]
This completes the proof of Proposition \ref{prop_cont}.

\section{Energy expansion: The $C^0$-estimates}\label{sec_expansion_1}
The main task of this section is to prove that estimates \eqref{energy_exp} holds in the $C^0$-sense.
We recall that the function  $V\dt$ is defined in \eqref{def_V}  and the function $\phi^{\ep}\dt$ is given in  Proposition \ref{prop_cont}. For the sake of brevity, we denote $V = V\dt$ and $\phi =\phi^{\ep}\dt.$
We decompose the reduced functional into three parts
\begin{multline*}
\wje(\md, \mt) = \(J_{\ep}(V\dt + \phi\dt^{\ep}) - J_{\ep}(V\dt)\) + \({1 \over 2} \intom a(x)|\nabla V\dt|^2 dx - {1 \over p + 1} \intom a(x) |V\dt|^{p+1}dx\)\\
+ \( {1 \over p + 1} \intom a(x) |V\dt|^{p+1} dx - {1 \over p + 1 -\ep} \intom a(x) |V\dt|^{p+1-\ep}dx\)
\end{multline*}
and we estimate each of them. The $C^0$-estimate will follow by the three lemmata Lemma \ref{lemma_C^0_1}, Lemma \ref{lemma_C^0_2} and Lemma \ref{lemma_C^0_3}.

\begin{lemma}\label{lemma_C^0_1}
It holds true that
\begin{equation}\label{s1}
J_{\ep}(V + \phi) - J_{\ep}(V) = o(\ep).
\end{equation}
\end{lemma}
\begin{proof}
Using Taylor's theorem and the fact that $J'_{\ep}(V+\phi)[\phi] = 0$, we get
\[J_{\ep}(V + \phi) - J_{\ep}(V) = - \int_0^1 t J_{\ep}''(V + t\phi)[\phi,\phi] dt.\]
On the other hand, since $\|\phi\| = o(\sqrt{\ep})$,
\[ \left|J_{\ep}''(V + t\phi))[\phi,\phi]\right| \le C\(\intom a |\nabla \phi|^2 + \sik \intom a U_i^{p-1-\ep} \phi^2 + \intom a |\phi|^{p+1-\ep}\) = o(\ep)\]
for some $C > 0$. Therefore   \eqref{s1} follows.
\end{proof}

It is useful to introduce the following   constants:
\begin{align}
a_1 &= \alpha_n^{p+1} \int_{\rn} {1 \over (1 + |y|^2)^n}dy, \label{a_1} \\
a_2 &= \alpha_n^{p+1} \int_{\rn} {1 \over (1 + |y|^2)^{n+2 \over 2}}dy, \label{a_2}\\
a_3 &= \alpha_n^{p+1} \int_{\rn} {1 \over (1 + |y|^2)^n} \log{\alpha_n \over (1+|y|^2)^{n-2 \over 2}}dy .
\end{align}
Here, $\alpha_n = [n(n-2)]^{n-2 \over 4}$.

\begin{lemma}\label{lemma_C^0_2}
It holds true that
\begin{equation}\label{s23}
\begin{aligned}
&~{1 \over 2} \intom a(x)|\nabla V\dt|^2 dx - {1 \over p + 1} \intom a(x) |V\dt|^{p+1}dx\\
&= \({1 \over 2} - {1 \over p+1}\) ka_1 \(a(\xi_0) + \partial_{\nu}a(\xi_0) t \ep\) + a(\xi_0) \left[ {a_2 \over 2} \({d_1 \over 2t}\)^{n-2} + \sum_{i=1}^{k-1} \({d_{i+1} \over d_i} \)^{n-2 \over 2} {\alpha_n^{p+1}|B^n| \over (1 + s_i^2)^{n-2 \over 2}} \right] \ep + o(\ep).\\
\end{aligned}
\end{equation}
\end{lemma}
\begin{proof}
Using  the definition of the annuli $A_i$ ($i = 1, \cdots, k$) in \eqref{A_l}, we write
\begin{equation}\label{s2}
\begin{aligned}
{1 \over 2} \intom a|\nabla V|^2 &= {1 \over 2} \slk \intom a |\nabla PU_l|^2 + \sum_{l < i} (-1)^{l+i} \intom a \nabla PU_l \cdot \nabla PU_i \\
&= {1 \over 2} \slk \left[\intal a U_l^{p+1} + \intom a U_l^p (PU_l -U_l) + \int_{\Omega \setminus A_l} a U_l^{p+1} - \intom (\nabla a \cdot \nabla PU_l) PU_l \right]\\
&\ + \sum_{l < i} (-1)^{l+i} \left[\intal a U_l^pU_i + \intom a U_l^p (PU_i -U_i) + \int_{\Omega \setminus A_l} a U_l^p U_i - \intom (\nabla a \cdot \nabla PU_l) PU_i \right]
\end{aligned}
\end{equation}
and
\begin{equation}\label{s3}
\begin{aligned}
&~{1 \over p + 1} \intom a|V|^{p+1} dx\\
&= {1 \over p + 1} \slk \intal a \bigg|\sik (-1)^{i+1} PU_i \bigg|^{p+1} + {1 \over p + 1} \int_{\Omega \setminus B(\xi_k, \rho \ep)} a\bigg|\sik (-1)^{i+1} PU_i \bigg|^{p+1}\\
&= {1 \over p + 1} \slk \intal a \bigg(\bigg|(-1)^{l+1} PU_l + \sum_{i \ne l} (-1)^{i+1} PU_i \bigg|^{p+1} - U_l^{p+1} \bigg) + {1 \over p + 1} \slk \intal a U_l^{p+1} + o(\ep)\\
&= \slk \bigg[ {1 \over p + 1} \intal a U_l^{p+1} + \intal aU_l^p(PU_l - U_l) \bigg] + \sum_{i \ne l} (-1)^{i+l} \bigg[ \intal aU_l^p U_i + \intal aU_l^p (PU_i-U_i) \bigg]\\
&~+ p\int_0^1 (1-\theta) \intal a \bigg| (-1)^{l+1} U_l + \theta\Big[ (-1)^{l+1}(PU_l-U_l) + \sum_{i \ne l} (-1)^{i+1} PU_i\Big]\bigg|^{p-1}\\
&\hspace{200pt}\bigg((-1)^{l+1}(PU_l-U_l) + \sum_{i \ne l} (-1)^{i+1} PU_i\bigg)^2 dx d\theta + o(\ep).
\end{aligned}
\end{equation}

First of all, we claim that
\begin{equation}\label{s3_1}
\sum_{i \ne l} (-1)^{i+l} \intal aU_l^p U_i = 2 \sum_{l < i} (-1)^{l+i} \intal a U_l^pU_i + o(\ep).
\end{equation}
 Indeed, suppose $l > i$. By the fact that $-\Delta PU_i = U_i^p$ in $\Omega$ and $PU_i = 0$ on $\po$, it follows that
\begin{equation}\label{s3_-1}
\begin{aligned}
\intal U_l^pU_i &= \intom \nabla PU_l \cdot \nabla PU_i - \intom U_l^p(PU_i - U_i) - \int_{\Omega \setminus A_l} U_l^pU_i\\
&= \intai U_i^pU_l + \intom U_i^p(PU_l - U_l) + \int_{\Omega \setminus A_i} U_i^pU_l - \intom U_l^p(PU_i - U_i) - \int_{\Omega \setminus A_l} U_l^pU_i.
\end{aligned}
\end{equation}
By Lemma A.1 and A.2 (see  also \eqref{s23_3}) we deduce
\[\intom U_i^p(PU_l - U_l),\ \intom U_l^p(PU_i - U_i) = o(\ep),\]
\begin{equation}\label{s3_0}
\begin{aligned}
&\ \int_{\Omega \setminus A_l} U_i^pU_l\\
&\le \( {\delta_l \over \delta_i} \)^{n-2 \over 2} \left[ \int_{B\(0, \sqrt{\delta_{l-1} \over \delta_l}\)^c} + \int_{B\(0, \sqrt{\delta_{l+1} \over \delta_l}\)} \right]{\alpha_n^{p+1}  \over (1 + |y - s_i \nu(\xi_0)|^2)^{n+2 \over 2}} {1 \over |y - (\delta_l /\delta_i) s_l \nu(\xi_0)|^{n-2}}\ dy\\
&= o(\ep)
\end{aligned}
\end{equation}
and
\begin{align*}
&\ \int_{\Omega \setminus A_l} U_l^pU_i\\
&\le \left[ \int_{B\(\xi_k, \sqrt{\delta_{l-1}\delta_l}\)^c} + \int_{B\(\xi_k, \sqrt{\delta_l\delta_{l+1}}\)} \right] {\alpha_n^{p+1} \delta_l^{n+2 \over 2} \over (\delta_l^2 + |x -\xi_k -s_l\delta_l\nu(\xi_0)|^2)^{n+2 \over 2}}{\delta_i^{n-2 \over 2} \over (\delta_i^2 + |x -\xi_k -s_i\delta_i\nu(\xi_0)|^2)^{n-2 \over 2}}\ dx\\
&\le \alpha_n^{p+1} \(\delta_l \over \delta_i \)^{n-2 \over 2} \left[ \int_{B\(0, \sqrt{\delta_{l-1} \over \delta_l}\)^c} + \int_{B\(0, \sqrt{\delta_{l+1} \over \delta_l}\)} \right] {1 \over (1 + |y -s_l\nu(\xi_0)|^2)^{n+2 \over 2}}\ dy = o(\ep).
\end{align*}
Therefore, equation \eqref{s3_-1} can be rewritten as
\begin{equation}\label{symm}
\intal U_l^pU_i = \intai U_i^pU_l + o\(\ep\).
\end{equation}
Moreover, we have the estimates
\begin{equation}\label{symm2}
\intal (a(x)-a(\xi_0)) U_l^p U_i dx,\ \intai (a(x)-a(\xi_0)) U_i^p U_l dx = o(\ep).
\end{equation}
By \eqref{symm} and \eqref{symm2}, we deduce  that
\begin{align*}
\intal aU_l^p U_i &= \left[a(\xi_0) \intal U_l^p U_i + \intal (a(x)-a(\xi_0)) U_l^p U_i dx\right]\\
&= a(\xi_0) \intal U_l^p U_i + o(\ep) = a(\xi_0) \intal U_i^p U_l + o(\ep) = \intai aU_i^p U_l + o(\ep),
\end{align*}
which in particular implies \eqref{s3_1}.

Next, we claim that the term $I := p\int_0^1(1-\theta)\intal \cdots$ in \eqref{s3} is of order $o(\ep)$.
Indeed, we first remark that
\begin{equation}\label{s3_1.5}
\intal |PU_l-U_l|^{p+1} = O\( \ep^{n \over n-2} \) \quad \text{and} \quad \intal U_i^{p+1} = O\( \ep^{n \over n-2}\) \quad \text{for } i \ne l.
\end{equation}
where the first equality is obtained in the proof of Lemma \ref{lemma_pre_3.5}  and the second one is deduced in (6.19) of \cite{MP}. Moreover, by \eqref{s23_3} and \eqref{s23_4}, we deduce
\[\intal U_l^p |PU_l-U_l|,\ \intal U_l^p U_i = O(\ep) \quad \text{if } i \ne l.\]
By these estimates, we get
\begin{equation}\label{s3_2}
\begin{aligned}
I &\le  C \left[ \intal U_l^{p-1}|PU_l-U_l|^2 + \intal |PU_l-U_l|^{p+1} + \sum_{i \ne l} \intal U_i^{p-1}|PU_l-U_l|^2 + \sum_{i \ne l}\intal |PU_l-U_l|^{p-1}U_i^2 \right]\\
&\ + C \left[ \sum_{i \ne l}\intal U_l^{p-1}U_i^2 + \sum_{i \ne l}\sum_{j \ne l} \intal U_i^{p-1}U_j^2 \right]\\
& \le C \left[ \(\intal U_l^p |PU_l-U_l| \)^{4 \over n+2}\(\intal |PU_l-U_l|^{p+1}\)^{n-2 \over n+2} + \intal |PU_l-U_l|^{p+1} \right.\\
&\ \left. + \sum_{i \ne l} \(\intal U_i^{p+1}\)^{2 \over n}\(\intal |PU_l-U_l|^{p+1}\)^{n-2 \over n} + \sum_{i \ne l}\(\intal |PU_l-U_l|^{p+1}\)^{2 \over n}\(\intal U_i^{p+1}\)^{n-2 \over n} \right]\\
&\ + C \left[\(\intal U_l^p U_i\)^{4 \over n+2}\(\intal U_i^{p+1}\)^{n-2 \over n+2} + \sum_{i \ne l}\sum_{j \ne l} \(\intal U_i^{p+1}\)^{2 \over n} \(\intal U_j^{p+1}\)^{n-2 \over n} \right]\\
& \le C \left[ \ep^{n+4 \over n+2} + \ep^{n \over n-2} \right] + C\left[ \ep^{n+4 \over n+2} + \ep^{n \over n-2} \right] = O\( \ep^{n+4 \over n+2}\) = o(\ep)
\end{aligned}
\end{equation}
for some constant $C > 0$.

Finally, by \eqref{symm2}, \eqref{s3_1} and \eqref{s3_2}, we get
\begin{equation}\label{s3_3}
{1 \over p + 1} \intom a|V|^{p+1} dx = \slk \bigg[ {1 \over p + 1} \intal a U_l^{p+1} + \intal aU_l^p(PU_l - U_l) \bigg] + 2\sum_{i < l} (-1)^{i+l} \intal aU_l^p U_i+ o(\ep).
\end{equation}

Moreover, by \eqref{s2}, \eqref{s3_0}, \eqref{symm2}, \eqref{s3_1.5}, \eqref{s3_3} and the estimate
\[\intom (\nabla a \cdot \nabla PU_l) PU_i = o(\ep) \quad \text{for } i,\ l = 1, \cdots, k\]
which is easily deduced by   Lemma \ref{lemma_pre_3}, we find  that
\begin{equation}\label{s23_1}
\begin{aligned}
&\ {1 \over 2} \intom a|\nabla V|^2 - {1 \over p + 1} \intom a|V|^{p+1} dx \\
&= \({1 \over 2}- {1 \over p+1}\) \slk \intal a U_l^{p+1} - {1 \over 2} \slk \intom a U_l^p (PU_l -U_l) + \sum_{l < i} (-1)^{l+i+1} \intal a U_l^pU_i + o(\ep).
\end{aligned}
\end{equation}

Now, we estimate each term   in the right-hand side of the above equality.
Firstly, we write the first term as
\[\intal a U_l^{p+1} = a(\xi_0) \intal U_l^{p+1} + \intal \(a(x)-a(\xi_0)\) U_l^{p+1} dx\]
and then we estimate
\begin{align*}
a(\xi_0) \intal U_l^{p+1}
&= a(\xi_0) a_1 - a(\xi_0) \alpha_n^{p+1} \left[ \int_{B\(\xi_k, \sqrt{\delta_{l-1} \over \delta_l}\)^c} + \int_{B\(\xi_k, \sqrt{\delta_{l} \over \delta_{l+1}}\)}\right]
{\delta_l^n \over (\delta_l^2 + |x - \xi_k - \delta_ls_l\nu(\xi_0)|^2)^n}\ dx\\
&= a(\xi_0) a_1 - a(\xi_0) \alpha_n^{p+1} \left[ \int_{B\(0, \sqrt{\delta_{l-1} \over \delta_l}\)^c} + \int_{B\(0, \sqrt{\delta_{l+1} \over \delta_{l}}\)}\right] {dy \over (1 + |y - s_l\nu(\xi_0)|^2)^n}\\
&= a(\xi_0) a_1 +o(\ep)
\end{align*}
and
\[\intal \(a(x)-a(\xi_0)\) U_l^{p+1} dx = \partial_{\nu}a(\xi_0) a_1 t \ep + o(\ep).\]
(cf. \cite[Lemma C.1]{ACP}). This shows that
\begin{equation}\label{s23_2}
\intal a U_l^{p+1} = a(\xi_0) a_1 + \partial_{\nu}a(\xi_0) a_1 t \ep + o(\ep).
\end{equation}
Secondly, by Lemma \ref{lemma_pre_1} and Lemma \ref{lemma_pre_2} (using the the mean value theorem) we deduce
\begin{equation}\label{s23_3}
\begin{aligned}
\intom a U_l^p (PU_l -U_l) &= -\alpha_n \delta_l^{n-2 \over 2} \intom aU_l^p H(\cdot,\xi_l) + o(\ep)\\
& = -\alpha_n^{p+1} \delta_l^{n-2} a(\xi_0) \int_{B(0, \rho \ep \delta_l^{-1})} {1\over (1 + |y|^2)^{n+2 \over 2}} \left[{1 \over (2\ep t)^{n-2}} + O\({\delta_l(1+|y|) \over \ep^{n-1}}\)\right]dy + o(\ep)\\
&= -\delta_{l1} \cdot \left\{a(\xi_0) a_2 \({d_1 \over 2t}\)^{n-2}\right\} \cdot \ep + o(\ep)
\end{aligned}
\end{equation}
where $\delta_{ij}$ is the Kronecker delta (cf. \cite[Lemma C.2 (64)]{ACP}).

 Finally, for $l < i$, we get
\begin{equation}\label{s23_4}
\begin{aligned}
&\ \intal a U_l^pU_i\\
&= a(\xi_0) \intal U_l^pU_i + \intal \(a(x)-a(\xi_0)\) U_l^pU_i dx\\
&= \( {\delta_i \over \delta_l} \)^{n-2 \over 2} \int_{B\(0, \sqrt{\delta_{l-1} \over \delta_{l}}\) \setminus B\(0, \sqrt{\delta_{l+1} \over \delta_{l}}\)} {a(\xi_0) \alpha_n^{p+1}  \over (1 + |y - s_l \nu(\xi_0)|^2)^{n+2 \over 2}} {dy \over [ (\delta_i/\delta_l)^2 + |y - (\delta_i /\delta_l) s_i \nu(\xi_0)|^2]^{n-2 \over 2}}+ o(\ep)\\
&= \delta_{i(l+1)} a(\xi_0) \({d_{l+1} \over d_l} \)^{n-2 \over 2} F(s_l) \ep + o(\ep).
\end{aligned}
\end{equation}

Here
\begin{equation}\label{F_s}
F(s) := \alpha_n^{p+1} \int_{\rn} {1 \over (1 + |y|^2)^{n+2 \over 2}} {1 \over |y + s \nu(\xi_0)|^{n-2}} dy= \alpha_n^{p+1}|B^n| {1 \over (1 + s^2)^{n-2 \over 2}}.
\end{equation}
The last equality follows from the fact that $U = U_{1,0}$ solves the equation $-\Delta U=U^p$ in $\rn$ and so it can be rewritten using the Green's representation formula
\[U(x)= {1 \over n(n-2)|B^n|}\int_{\rn} U^{p}(y){1 \over |y-x|^{n-2}} dy,\]
which implies $F(s)=\alpha_n^p|B^n| U\left(s\nu(\xi_0)\right).$

By combining \eqref{s23_2}, \eqref{s23_3} and \eqref{s23_4} with \eqref{s23_1}, estimate   \eqref{s23} follows.
\end{proof}

\begin{lemma}\label{lemma_C^0_3}
It holds true that
\begin{equation}\label{s4}
\begin{aligned}
&\ {1 \over p + 1} \intom a|V|^{p+1} - {1 \over p + 1 -\ep} \intom a |V|^{p+1-\ep}\\
&= - a(\xi_0) {k(n+k-2) \over 2(p+1)}\cdot a_1 \ep \log \ep + a(\xi_0) \left[{k a_3 \over p+1} -{k a_1 \over (p+1)^2} - {(n-2)^2\over 4n}\cdot a_1 \sik \log d_i\right]\ep + o(\ep).
\end{aligned}
\end{equation}
\end{lemma}
\begin{proof}
By the Taylor expansion we deduce
\begin{equation}\label{s4_1}
\begin{aligned}
&{1 \over p + 1} \intom a|V|^{p+1} - {1 \over p + 1 -\ep} \intom a |V|^{p+1-\ep}\\
&= \left[{1 \over p+1} \intom a \left| \sik (-1)^{i+1} PU_i \right|^{p+1} \log \left| \sik (-1)^{i+1} PU_i \right|
- {1 \over (p+1)^2} \intom a \left| \sik (-1)^{i+1} PU_i \right|^{p+1} \right] \ep + o(\ep).
\end{aligned}
\end{equation}
Arguing as in the proof of the previous lemma we get
\begin{equation}\label{step_4_1}
\intom a \left| \sik (-1)^{i+1} PU_i \right|^{p+1} =  a(\xi_0) ka_1 + o(1).
\end{equation}
Moreover, we have
\begin{equation}\label{step_4_2}
\begin{aligned}
&\ \intom a \left| \sik (-1)^{i+1} PU_i \right|^{p+1} \log \left| \sik (-1)^{i+1} PU_i \right|\\
&= \sum_{j=1}^k\int_{A_j} a \left| \sik (-1)^{i+1} U_i \right|^{p+1} \log \left| \sik (-1)^{i+1} U_i \right| + o(1) \\
&= - a(\xi_0) {k(n+k-2) \over 2(p+1)}\cdot a_1 \log \ep + a(\xi_0) \left[{k a_3 \over p+1} - {(n-2)^2\over 4n}\cdot a_1 \sik \log d_i\right] + o(1).
\end{aligned}
\end{equation}
By combining \eqref{s4_1}, \eqref{step_4_1} and \eqref{step_4_2},  \eqref{s4} follows.

\medskip
Let us prove \eqref{step_4_2}.

To get the first equality, it is sufficient to show that
\begin{equation}\label{step_3_2_1}
\intom a \left| \sik (-1)^{i+1} PU_i \right|^{p+1} \log \left| \sik (-1)^{i+1} PU_i \right| = \intom a \left| \sik (-1)^{i+1} U_i \right|^{p+1} \log \left| \sik (-1)^{i+1} U_i \right| + o(1)
\end{equation}
and
\begin{equation}\label{step_3_2_2}
\int_{\Omega \setminus B(\xi_k, \rho\ep)} a \left| \sik (-1)^{i+1} U_i \right|^{p+1} \log \left| \sik (-1)^{i+1} U_i \right| = o(1).
\end{equation}

\medskip
If we write
\[V := \sik (-1)^{i+1} PU_i, \quad E := \sik (-1)^{i+1} (U_i-PU_i) \quad \text{and} \quad g(s) := |s|^{p+1}\log|s| \text{ for } s \ne 0,\]
then we see that
\begin{align*}
&\ \intom a \cdot |g(V+E) - g(V)| dx\\
&\le C \intom \int_0^1\(|V+\theta E|^{p + \sigma} + |V+\theta E|^{p - \sigma} + |V+\theta E|^p\) \cdot |E|\ d\theta dx
\quad \text{(by (a1) and \eqref{cont_2})}\\
&\le C \left[\intom \(|V|^{p + \sigma} + |V|^{p - \sigma} + |V|^p\) \cdot |E| dx + \intom \(|E|^{p + \sigma} + |E|^{p - \sigma} + |E|^p\) dx \right]\\
&= o(1) \quad \text{(by the H\"older inequality and Lemma \ref{lemma_pre_3.5})}
\end{align*}
for some constant $C > 0$. This proves \eqref{step_3_2_1}.

Furthermore, denoting $\widetilde{V} := \sum\limits_{i=1}^k (-1)^{i+1} U_i$, we have
\begin{align*}
\int_{\Omega \setminus B(\xi_k, \rho\ep)} a g\big(\widetilde{V}\big) &\le C \int_{\Omega \setminus B(\xi_k, \rho\ep)} \(|\widetilde{V}|^{p + \sigma} + |\widetilde{V}|^{p - \sigma}\) \le C \sik \int_{\Omega \setminus B(\xi_k, \rho\ep)} \(U_i^{p + \sigma} + U_i^{p - \sigma}\)\\
&\le C \sik \({\delta_i \over \ep}\)^n \left[ \({\delta_i \over \ep^2}\)^{{n-2 \over 2}\sigma} + \({\delta_i \over \ep^2}\) ^{-{n-2 \over 2}\sigma}\right] = o(1),
\end{align*}
which implies \eqref{step_3_2_2}.

Finally, the second equality can be obtained as in (6.39) in \cite{MP}.
\end{proof}

From Lemma \ref{lemma_C^0_1}, \ref{lemma_C^0_2} and \ref{lemma_C^0_3}, we conclude that estimate \eqref{energy_exp} is true in the $C^0$-sense.

\section{Energy expansion: The $C^1$-estimates}\label{sec_expansion_2}
In this section, we will deduce that \eqref{energy_exp} holds $C^1$-uniformly on compact subsets of the admissible set $\Lambda$.

\medskip
Let us denote again $V = V\dt$ and $\phi = \phi^{\ep}\dt$ for the sake of simplicity.
We need to prove that for $\md := (d_1, \cdots, d_k) \in (0, +\infty)^k$ and $\mt := (t, s_1, \cdots, s_{k-1}) \in (0, +\infty) \times \mathbb{R}^{k-1}$,
\begin{equation}\label{s5}
\pr \wje(\md, \mt) = \pr \Phi(\md, \mt) \ep + o(\ep)
\end{equation}
$C^0$-uniformly on compact sets of $\Lambda$ where $\wje$ and $\Phi$ are defined in \eqref{red_energy} and \eqref{phi}, respectively,
and $r$ is one of $d_1, \cdots,\ d_k,\ t,\ s_1, \cdots,\ s_{k-2}$ and $s_{k-1}$.

\subsection{The case $r = d_l\ (l = 1, \cdots, k)$ or $r = s_l\ (l = 1, \cdots, k-1)$}
We decompose $\pr \wje(\md, \mt)$ into
\[\pr \wje(\md, \mt) = J_{\ep}'(V)(\pr V) + [J_{\ep}'(V + \phi)- J_{\ep}'(V)]\pr V + J_{\ep}'(V+\phi)(\pr \phi)\]
and estimate each term.

\begin{lemma}\label{lemma_C^1_1}
It is satisfied that
\begin{equation}\label{s5_1}
J_{\ep}'(V)(\pr V) = \pr \Phi(\md, \mt) \ep + o(\ep) \quad \text{for } r = d_1,\cdots, d_k, s_1, \cdots, s_{k-1}.
\end{equation}
\end{lemma}
\begin{proof}
Set $p = (n+2)/(n-2)$. We split $J'_\ep(V)(\pr V)$ as
\begin{align*}
&\ J'_\ep(V)(\pr V)\\
&= \intom a \nabla V \cdot \nabla (\pr V) - \intom a |V|^{p-1-\ep} V (\pr V)\\
&= \left[\sik (-1)^{i+1} \intom a \nabla PU_i \cdot \nabla (\pr V) - \intom a |V|^{p-1} V (\pr V) \right]
+ \left[\intom a |V|^{p-1} V (\pr V) - \intom a |V|^{p-1-\ep} V (\pr V) \right]\\
&= \intom a \(\sik (-1)^{i+1} U_i^p - |V|^{p-1}V \) \cdot (\pr V) + \sik (-1)^i \intom (\nabla a \cdot \nabla PU_i) (\pr V)\\
& \hspace{280pt} + \left[\intom a \(|V|^{p-1} V - |V|^{p-1-\ep} V\) \cdot (\pr V) \right]\\
&=: T_r^1 + T_r^2 + T_r^3
\end{align*}
and estimate each $T_r^i$ ($i = 1,2,3$).

Suppose that $r = d_l$ for some $l = 1, \cdots, k$. Note that in this case
\[\pr V = \pdl V = (-1)^{l+1} \pdl PU_l = (-1)^{l+1} \ep^{n-1 + 2(l-1) \over n-2} \cdot P\(\psi_l^0 + s_l \psi_l^n\)\]
where $P: D^{1,2}(\rn) \to \ho$ is the projection operator given by \eqref{proj} and $\psi_l^j := \psi_{\delta_l, \xi_l}^j$ ($j = 0,\ n$) are functions defined as \eqref{pdxn} and \eqref{pdxj}. By simple manipulation, we get
\begin{align*}
T_{d_l}^1 &= \intal a \(\sik (-1)^{i+1} U_i^p - |V|^{p-1}V \) \cdot (-1)^{l+1} \pdl PU_l + o(\ep)\\
&= \intal a \(\left|(-1)^{l+1} U_l\right|^{p-1}(-1)^{l+1} U_l - |V|^{p-1}V \) \cdot (-1)^{l+1} \pdl PU_l+ o(\ep).
\end{align*}
On the other hand, 
by adapting the way to estimate $I$ in the $C^0$-estimation and using \eqref{pre_5_3}, we can deduce that
\[\left| \ep^{n-1 + 2(l-1) \over n-2} \intal a\(\left|(-1)^{l+1} U_l\right|^{p-1}(-1)^{l+1} U_l - |V|^{p-1}V \) \cdot (-1)^{l+1} \big(P\psi_l^j - \psi_l^j\big) \right| = O\(\ep^{n+4 \over n+2}\).\]
Thus by the mean value theorem
\begin{align*}
T_{d_l}^1 &= \intal a \(\left|(-1)^{l+1} U_l\right|^{p-1}(-1)^{l+1} U_l - |V|^{p-1}V \) \cdot (-1)^{l+1} \pdl U_l+ o(\ep)\\
&= p \intal aU_l^{p-1} (U_l- PU_l)\pdl U_l + \sum_{i \ne l} (-1)^{i+l+1} p \intal aU_l^{p-1} PU_i \cdot \pdl U_l + o(\ep)
\end{align*}
From Lemma \ref{lemma_pre_1} and \ref{lemma_pre_2}, it follows that
\begin{align*}
p \intal aU_l^{p-1} (U_l- PU_l)\pdl U_l &= p \intal aU_l^{p-1} (U_l- PU_l) d_l^{-1} \delta_l \psi_l^0 + p \intal aU_l^{p-1} (U_l- PU_l) d_l^{-1}s_l \delta_l \psi_l^n\\
& = \left[\delta_{l1} a(\xi_0) {a_2 \over 2} \partial_{d_1} \({d_1 \over 2t}\)^{n-2} \ep + o(\ep)\right] + o(\ep).
\end{align*}
Furthermore, for $l < i$, we obtain by applying Lemma \ref{lemma_pre_00} in particular that
\begin{align*}
&\ p \intal aU_l^{p-1} PU_i \cdot \pdl U_l\\
&= p \intal aU_l^{p-1} U_i\pdl U_l + o(\ep) = \intal a \(\pdl U_l^p\) U_i + o(\ep)
= \pdl\(\intal aU_l^pU_i\) - \left.\partial_{d_m}\(\int_{A_m} aU_l^pU_i\)\right|_{l=m} + o(\ep)\\
&= \delta_{i(l+1)} a(\xi_0)\alpha_n^{p+1}\ep \cdot \pdl\Bigg[\({d_{l+1} \over d_l}\)^{n-2 \over 2}\int_{B\(0, \sqrt{\delta_{l-1} \over \delta_l}\) \setminus B\(0, \sqrt{\delta_{l+1} \over \delta_l}\)} {1 \over (1+|y - s_l \nu(\xi_0)|^2)^{n+2 \over 2}}\\
&\hspace{160pt} \times {1 \over \big[ (d_{l+1}/d_l)^2 \cdot \ep^{4 \over n-2} + \big|y - (d_{l+1}/d_l) \ep^{2 \over n-2} s_{l+1}\nu(\xi_0)\big|^2 \big]^{n-2 \over 2}}\ dy \Bigg] + o(\ep)\\
&= \delta_{i(l+1)} a(\xi_0) \pdl \({d_{l+1} \over d_l}\)^{n-2 \over 2} F(s_l)\ep + o(\ep)
\end{align*}
where we set $d_{k+1} = 0$ and the function $F$ is defined in \eqref{F_s}.
If $l > i$, through the procedure changing the order of $i$ and $l$ that was conducted in computing \eqref{s3} (see \eqref{s3_-1} and the following computations), we can see
\[p \intal aU_l^{p-1} PU_i \cdot \pdl U_l = a(\xi_0) \delta_{i(l-1)}\pdl \({d_l \over d_{l-1}}\)^{n-2 \over 2} F(s_{l-1}) \ep + o(\ep),\]
letting $F(s_0) = 0$. As a result, it holds that
\begin{equation}\label{tdl_1}
T_{d_l}^1 = a(\xi_0) \left[\delta_{l1} {a_2 \over 2} \partial_{d_1} \({d_1 \over 2t}\)^{n-2} + \pdl \({d_{l+1} \over d_l}\)^{n-2 \over 2} F(s_l) + \pdl \({d_l \over d_{l-1}}\)^{n-2 \over 2} F(s_{l-1})  \right] \ep + o(\ep).
\end{equation}

Employing Lemma \ref{lemma_pre_3}, we can easily show that
\begin{equation}\label{tdl_2}
T_{d_l}^2 = o(\ep),
\end{equation}
so it suffices to compute $T_{d_l}^3$. Clearly
\[T_{d_l}^3 = \ep \intom a|V|^{p-1}V \log|V| \cdot (-1)^{l+1}\ep^{n-1 + 2(l-1) \over n-2} \(\psi_l^0+s_l \psi_l^n\) + o(\ep).\]
Also, utilizing
\begin{equation}\label{int_y_zero}
\int_{\rn} {|y|^2-1 \over (1+|y|^2)^{n+1}} dy = \int_{\rn} {y_n \over (1+|y|^2)^{n+1}} dy = 0,
\end{equation}
Lemma \ref{lemma_pre_00} and performing a similar computation to the derivation of \eqref{step_4_2}, we find
\begin{align*}
&\ \ep \intom a|V|^{p-1}V \log|V| \cdot (-1)^{l+1}\ep^{n-1 + 2(l-1) \over n-2} \psi_l^0\\
&= \ep \sum_{j=1}^k \int_{A_j} a\left|\sik (-1)^{i+1}U_i\right|^{p-1}\(\sik (-1)^{i+1}U_i\) \log\left|\sik (-1)^{i+1}U_i\right| \cdot (-1)^{l+1}\ep^{n-1 + 2(l-1) \over n-2} \psi_l^0 + o(\ep)\\
&= \ep \sum_{j=1}^k \int_{A_j} a U_j^p \log U_j \cdot (-1)^{l+j}\ep^{n-1 + 2(l-1) \over n-2} \psi_l^0 + o(\ep)\\
&= {1 \over p+1}\ a(\xi_0) \ep d_l^{-1}\delta_l \int_{A_l} \(\partial_{\delta_l} U_l^{p+1}\) \log U_l + o(\ep)\\
&= {1 \over p+1}\ a(\xi_0) \ep d_l^{-1}\delta_l \cdot \left[\partial_{\delta_l} \(\int_{A_l} U_l^{p+1} \log U_l\) - \int_{A_l} U_l^p \psi_l^0 - \left. \partial_{\delta_m} \(\int_{A_m} U_l^{p+1} \log U_l \) \right|_{m=l} \right] + o(\ep)\\
&= - {1 \over p+1} \cdot {n-2 \over 2} \cdot a(\xi_0) d_l^{-1} \ep \int_{B\(0, \sqrt{\delta_{l-1} \over \delta_l}\) \setminus B\(0, \sqrt{\delta_{l+1} \over \delta_l}\)} U_l^{p+1} \log U_l + o(\ep)\\
&= - {(n-2)^2\over 4n} a(\xi_0) d_l^{-1} a_1 \ep + o(\ep)
\end{align*}
and
\[\ep \intom a|V|^{p-1}V \log|V| \cdot (-1)^{l+1}\ep^{n-1 + 2(l-1) \over n-2} \psi_l^n = a(\xi_0) \ep \int_{A_l} U_l^p \log U_l \cdot \delta_l \psi_l^n + o(\ep) = o(\ep).\]
Thus
\begin{equation}\label{tdl_3}
T_{d_l}^3 = - {(n-2)^2\over 4n} a(\xi_0) d_l^{-1} a_1 \ep + o(\ep).
\end{equation}

Combining \eqref{tdl_1}, \eqref{tdl_2} and \eqref{tdl_3}, we see that
\begin{align*}
&\ J'_\ep(V)(\pr V)\\
&= a(\xi_0) \left[\delta_{l1} {a_2 \over 2} \cdot \partial_{d_1} \({d_1 \over 2t}\)^{n-2} + \left\{\pdl \({d_{l+1} \over d_l}\)^{n-2 \over 2} {\alpha_n^{p+1}|B^n| \over (1 + s_l^2)^{n-2 \over 2}} + \pdl \({d_l \over d_{l-1}}\)^{n-2 \over 2} {\alpha_n^{p+1}|B^n|  \over (1 + s_{l-1}^2)^{n-2 \over 2}} \right\} \right] \ep\\
&\hspace{275pt} - {(n-2)^2\over 4n} a(\xi_0) a_1 \(\pdl \log d_l\) \ep + o(\ep)
\end{align*}
and hence \eqref{s5_1} is valid if $r = d_l$.

\medskip
The case $r = s_l$ for some $l = 1, \cdots, k-1$ can be dealt with in a similar way to the case $r = d_l$. Hence the proof follows.
\end{proof}

\begin{lemma}\label{lemma_C^1_2}
For any $r = d_1,\cdots, d_k, s_1, \cdots, s_{k-1}$, the following holds:
\[[J_{\ep}'(V + \phi)- J_{\ep}'(V)]\pr V = o(\ep).\]
\end{lemma}
\begin{proof}
We consider only when $r = d_l$ here. The case $r = s_l$ is similar. Expand
\begin{align*}
[J_{\ep}'(V + \phi)- J_{\ep}'(V)]\pdl V&= \left[\intom a \nabla \phi \cdot \nabla \pdl V - ap|V|^{p-1}\phi \pdl V\right]\\
&\ - \left[\intom a \left\{|V + \phi|^{p-1-\ep}(V+\phi) - |V|^{p-1-\ep}V - (p-\ep)|V|^{p-1-\ep}\phi \right\} \pdl V \right]\\
&\ + \left[\intom a\left\{p|V|^{p-1} - (p-\ep) |V|^{p-1-\ep}\right\} \phi \pdl V\right]\\
&=: I_1 + I_2 + I_3
\end{align*}
and study each summands.

Let us estimate $I_1$. We have
\[I_1 = \sik \(\intom a \nabla \phi \cdot \nabla \delta_i\(P\psi_i^0 + s_i P\psi_i^n\) - p \intom a |V|^{p-1} \phi \delta_i\(P\psi_i^0 + s_i P\psi_i^n\)\).\]
By \eqref{eq_of_psi} and \eqref{proj},
\begin{multline*}
\intom a \nabla \phi \cdot \nabla \(\delta_i P\psi_i^j\) - p \intom a |V|^{p-1} \phi \(\delta_i P\psi_i^j\) \\
= p \intom a \phi \(U_i^{p-1} - |V|^{p-1}\)\delta_i \psi_i^j - p \intom a\phi|V|^{p-1}\delta_i\big(P\psi_i^j-\psi_i^j\big) - \intom \nabla a \cdot \nabla \(\delta_iP\psi_i^j\) \phi
\end{multline*}
for $j = 0,\ n$, so it suffices to estimate three terms in the right-hand side of the above equality.
Notice that by \eqref{est_phi} and \eqref{s3_1.5}, we have
\begin{align*}
\int_{\Omega \setminus A_i} |\phi| \left|U_i^{p-1} - |V|^{p-1}\right| \big|\delta_i \psi_i^j\big|
&\le \|\phi\| \cdot \(\slk \big\|U_l^{p-1}\big\|_{L^{n \over 2}(\Omega \setminus A_i)}\) \cdot \big\|\delta_i\psi_i^j\big\|_{L^{p+1}(\Omega \setminus A_i)} \\
&= o(\sqrt{\ep})\cdot O(1) \cdot O(\sqrt{\ep}) = o(\ep)
\end{align*}
and
\begin{align*}
&\ \intai |\phi| \left|U_i^{p-1} - |V|^{p-1}\right| \big|\delta_i \psi_i^j\big|\\
&\le \chi \cdot C \intai |\phi| U_i \(|PU_i - U_i|^{p-1} + \sum_{l \ne i} U_l^{p-1}\) +  C \intai |\phi| U_i^{p-1} \(|PU_i-U_i| + \sum_{l \ne i} U_l\)\\
&\le \chi \cdot C \|\phi\| \cdot \|U_i\|_{L^{p+1}(A_i)}\(\left\||PU_i-U_i|^{p-1}\right\|_{L^{n \over 2}(A_i)} + \sum_{l \ne i} \big\|U_l^{p-1}\big\|_{L^{n \over 2}(A_i)}\)\\
&\ + C \|\phi\|\cdot \big\|U_i^{p-1}\big\|_{L^{n \over 2}(A_i)}\(\left\|PU_i-U_i\right\|_{L^{p+1}(A_i)} + \sum_{l \ne i} \|U_l\|_{L^{p+1}(A_i)}\)\\
&= \chi \cdot o(\sqrt{\ep}) \cdot O(1) \cdot O\(\ep^{2 \over n-2}\) + o(\sqrt{\ep})\cdot O(1) \cdot O(\sqrt{\ep}) = o(\ep)
\end{align*}
for some $C>0$ (see \cite[Lemma A.1]{MP}), where $\chi$ is a function such that $\chi = 0$ if $n \ge 6$ and $\chi = 1$ if $n \le 5$.
Furthermore, Lemma \ref{lemma_pre_5} implies
\[\intom |\phi||V|^{p-1}\delta_i\big|P\psi_i^j-\psi_i^j\big| \le \|\phi\|\cdot \|V^{p-1}\|_{L^{n \over 2}(\Omega)}\cdot\big\|\delta_i\big(P\psi_i^j -\psi_i^j\big)\big\|_{L^{p+1}(\Omega)} = o(\sqrt{\ep})\cdot O(1) \cdot O(\sqrt{\ep}) = o(\ep).\]
Finally, by applying Young's inequality (see Subsection \ref{appendix_young}) and \eqref{est_phi}, we observe that
\begin{equation}\label{s5_2}
\left| \intom \nabla a \cdot \nabla \(\delta_iP\psi_i^j\) \phi\right| \le C \delta_i \big\|U_i^{p-1} \psi_i^j\big\|_{L^{2n \over n+4 - \sigma}(\Omega)} \cdot \|\phi\|_{L^{2n \over n-2}(\Omega)} = O\(\delta_i^{1 - {\sigma \over 2}}\) \cdot o(\sqrt{\ep}) = o(\ep)
\end{equation}
where $\sigma > 0$ is a sufficiently small parameter. Therefore $I_1 = o(\ep)$.

Likewise, we can check that $I_2,\ I_3 = o(\ep)$ holds. (Refer to page 29-31 in \cite{MP}.)
\end{proof}

\begin{lemma}\label{lemma_C^1_3}
We have
\[J_{\ep}'(V+\phi)(\pr \phi) = o(\ep) \quad \text{for } r = d_1,\cdots, d_k, s_1, \cdots, s_{k-1}.\]
\end{lemma}
\begin{proof}
We can argue as in the derivation of (7.6) in \cite{MP}. Since we need a by-product that is derived during the proof of the lemma in the next subsection, we briefly sketch the proof.

\medskip
Equation \eqref{es1} reads as
\begin{equation}\label{C^1_1}
S(V + \phi) := -\text{div}(a \nabla (V+\phi)) - a|V+\phi|^{p-1-\ep}(V+\phi) = - \sum_{i=1}^k \left[c_{i0} \cdot \text{div}(a \nabla P\psi_i^0) + c_{in} \cdot \text{div}(a \nabla P\psi_i^n)\right].
\end{equation}
Testing \eqref{C^1_1} with the function $\pr \phi$ and using the fact $\phi \in K\pe$ where $K\pe$ is defined in \eqref{K_perp}, we get
\begin{equation}\label{C^1_2}
J_{\ep}'(V+\phi)(\pr \phi) = \sij c_{ij} \intom a \nabla P\psi_i^n \cdot \nabla(\pr \phi) = - \sij c_{ij} \intom a \nabla \big(\pr P\psi_i^n\big) \cdot \nabla \phi.
\end{equation}
On the other hand, testing \eqref{C^1_1} with the function $P\psi_l^m$ for any fixed $m = 1, \cdots, k$ and $l = 0,\ n$ and applying Lemma \ref{lemma_pre_4} and \ref{lemma_pre_3}, we can check that
\begin{equation}\label{C^1_3}
c_{ij} = o(\delta_i \sqrt{\ep}).
\end{equation}
Since Lemma \ref{lemma_pre_6} and \eqref{est_phi} imply that
\[\left|\intom a \nabla \big(\pr P\psi_i^n\big) \cdot \nabla \phi \right| \le C \big\|\pr P\psi_i^j\big\| \cdot \|\phi\| = o\(\delta_i^{-1}\sqrt{\ep}\)\]
for some $C > 0$, we get the result.
\end{proof}

To sum up, we deduce \eqref{s5} from Lemma \ref{lemma_C^1_1}, \ref{lemma_C^1_2} and \ref{lemma_C^1_3} if $r = d_l \ (l = 1, \cdots, k)$ or $r = s_l \ (l = 1, \cdots, k-1)$.

\subsection{The case $r = t$}
When $r = t$, we have
\[\pr V = \pt V = \sik (-1)^{i+1} \pt PU_i = \sik (-1)^{i+1} \ep \cdot P\psi_l^n.\]
Thus, unlike the previous case $r = d_l$ or $s_l$ where $\pr U_i = O\(\delta_i \(\psi_i^0 + \psi_i^n\)\) = O(U_i)$ holds, $\pt U_i = O(U_i)$ is not true anymore.
In fact, it turns out that this difference makes it hard to obtain \eqref{s5} in a direct way in this case.
Fortunately, we can borrow the idea from \cite{EMP} to overcome this problem, where the authors replaced the term,
in our setting, $\pt V(x) = \ep \sum\limits_{i=1}^k (-1)^{i+1} \partial_{(\xi_i)_n} V(x)$ with $\ep \sum\limits_{i=1}^k (-1)^{i+1} \partial_{x_n} V(x)$ ($x \in \Omega$) in the expansion of the reduced energy functional $\pt \wje$ and used a Pohozaev-type identity to estimate it.
Such an approach was also applied in \cite{MP} successfully.

\begin{lemma}\label{lemma_C^1_4}
We have
\[J_{\ep}'(V+\phi)(\pt V + \pt \phi) =  \pt \Phi(\md, \mt) \ep + o(\ep).\]
\end{lemma}
\begin{proof}
As the first step, let us compute $J_{\ep}'(V+\phi)(\pt V)$.
By utilizing \eqref{pre_5_2} and \eqref{pre_2_2}, we get
\[\ep |c_{ij}| \intom a U_i^{p-1} \big|\psi_i^j\big| \big|P\psi_l^n- \psi_l^n\big|,\quad \ep |c_{ij}| \intom a U_i^{p-1} \big|\psi_i^j\big| \big|\partial_{x_n}(PU_l - U_l)\big| = o\(\ep^{3 \over 2}\).\]
Also, the application of \eqref{C^1_3}, the proof of Lemma \ref{lemma_pre_8} and Young's inequality (see Subsection \ref{appendix_young}) gives
\begin{align*}
&\ \ep |c_{ij}| \intom \big|\nabla P\psi_i^j\big|\cdot\left|\partial_{(\xi_l)_n} PU_l + \pxn PU_l\right| \le o\(\delta_i\ep^{3 \over 2}\) \cdot O\({\delta_l^{n-2 \over 2} \over \ep^{n-1}}\) \intom \big|\nabla P\psi_i^j\big|\\
&\le o\(\delta_i \ep^{3 \over 2}\) \cdot O\({\delta_l^{n-2 \over 2} \over \ep^{n-1}}\) \intom \intom {1 \over |x-y|^{n-1}} \big(U_i^{p-1}\psi_i^j\big)(y) dy dx \le o\(\ep^{3 \over 2}\) \cdot O\({\delta_l^{n-2 \over 2} \over \ep^{n-1}}\) \|U_i^p\|_{L^1(\Omega)}\\
&= o\(\sqrt{\ep}\) \cdot O\({\delta_i^{n-2 \over 2}\delta_l^{n-2 \over 2} \over \ep^{n-2}}\)
= o\(\ep^{3 \over 2}\)
\end{align*}
Hence
\begin{align*}
&\ J_{\ep}'(V+\phi)(\pt V)\\
&= p \ep \sij c_{ij} \intom a U_i^{p-1} \psi_i^j \(\slk (-1)^{l+1} \partial_{(\xi_l)_n} PU_l\)
- \ep \sij c_{ij} \intom \nabla a \cdot \nabla P\psi_i^j \(\slk (-1)^{l+1} \partial_{(\xi_l)_n} PU_l\)\\
&= p\ep \sum_{i,j,l} (-1)^{l+1} c_{ij} \intom a U_i^{p-1} \psi_i^j \left[\big(P\psi_l^n- \psi_l^n\big) + \partial_{x_n}(PU_l - U_l) - \partial_{x_n} PU_l\right]\\
&\hspace{100pt} - \ep \sij c_{ij} \intom \nabla a \cdot \nabla P\psi_i^j \left[\slk (-1)^{l+1} \left\{\(\partial_{(\xi_l)_n} PU_l + \pxn PU_l\) - \pxn PU_l\right\}\right] \\
&= - \intom S(V+\phi) (\pxn V)\ep + o(\ep).
\end{align*}

To estimate $J_{\ep}'(V+\phi)(\pt \phi)$, we observe that \eqref{C^1_2} implies
\[J_{\ep}'(V+\phi)(\pt \phi) = \sij c_{ij} \intom a\big(\Delta \pt P\psi_i^j\big) \phi + \sij c_{ij} \intom \nabla a \cdot \nabla \big(\pt P\psi_i^j\big) \phi.\]
Since it holds that
\begin{align*}
\intom a \pt\big(\Delta P\psi_i^j\big) \phi &= - \intom pa \pt\big(U_i^{p-1}\psi_i^j\big)\phi = - p \ep \intom a \partial_{(\xi_i)_n}\big(U_i^{p-1} \psi_i^j\big)\phi = p \ep \intom a \pxn\big(U_i^{p-1} \psi_i^j\big)\phi\\
&= -p \ep \intom \pxn a \cdot U_i^{p-1} \psi_i^j \phi -p \ep \intom a U_i^{p-1} \psi_i^j (\pxn\phi) = -p \ep \intom a U_i^{p-1} \psi_i^j (\pxn\phi) + o\(\delta_i^{-1} \ep^{3 \over 2}\),
\end{align*}
and equation \eqref{C^1_3}, \eqref{est_phi} and Lemma \ref{lemma_pre_7} assert that
\begin{equation}\label{C^1_3.5}
\left|c_{ij} \intom \nabla a \cdot \nabla (\pt P\psi_i^j) \phi\right| \le |c_{ij}| \cdot \|\nabla a\|_{L^{\infty}(\Omega)} \cdot \big\|\nabla \pt P\psi_i^j\big\|_{L^{2n \over n+2}(\Omega)} \cdot \|\phi\| = o(\ep),
\end{equation}
(in fact, this is the only part we use the assumption $n \ge 4$ substantially; see Remark \ref{rmk_pre}), we deduce
\[J_{\ep}'(V+\phi)(\pt \phi) = -p \ep \sij c_{ij}\intom a U_i^{p-1} \psi_i^j (\pxn\phi) + o(\ep).\]
On the other hand, by multiplying \eqref{C^1_1} by $\pxn \phi$ and integrating the result over $\Omega$, we get
\[ \intom S(V+\phi) (\pxn \phi) = p\sij c_{ij} \intom a U_i^{p-1}\psi_i^j (\pxn\phi) + O\(\sij |c_{ij}| \cdot \big\|P\psi_i^j\big\| \cdot \|\phi\|\).\]
Thus using \eqref{C^1_3}, \eqref{est_phi} and Lemma \ref{lemma_pre_4}, we conclude that
\[J_{\ep}'(V+\phi)(\pt \phi) = - \intom S(V+\phi) (\pxn \phi) \ep + o(\ep).\]

Accordingly, if we set $u = V + \phi$,
\begin{align*}
J_{\ep}'(u) (\pt u) = - S(u) (\pxn u) \ep + o(\ep) &= \(\intom \text{div} (a \nabla u) \pxn u + \intom a |u|^{p-1-\ep}u \pxn u\) \ep + o(\ep)\\
&=: (K_1 + K_2)\ep + o(\ep).
\end{align*}

Let us estimate the term $K_2$: From \eqref{est_phi}, the proof of Lemma \ref{lemma_C^0_2} and (a3) (which implies $\pxn a(\xi_0) = \partial_{\nu}a(\xi_0)$), we find
\begin{align*}
K_2 &= {1 \over p+1-\ep} \intom a\(\pxn |u|^{p+1-\ep}\) = - {1 \over p+1-\ep} \intom \(\pxn a\) |V+\phi|^{p+1-\ep}\\
&= - {1 \over p+1-\ep} \intom \(\pxn a\) |V|^{p+1-\ep} + o(1) = - {1 \over p+1} \intom \(\pxn a\) |V|^{p+1} + o(1)\\
&= - {1 \over p+1} ka_1 \pxn a(\xi_0) + o(1) = - {1 \over p+1} ka_1 \partial_{\nu}a(\xi_0) + o(1)
\end{align*}
where $a_1$ is the quantity defined in \eqref{a_1}.

Next, we consider $K_1$: Write
\[K_1 = \intom (\nabla a \cdot \nabla u) (\pxn u) + \intom a \Delta u (\pxn u)
 = {1 \over 2} \intom (\pxn a) |\nabla u|^2 - {1 \over 2} \intpom a|\nabla u|^2 \nu_n dS =: K_{11} + K_{12}\]
where $\nu_n$ is the $n$-th component of the inward unit normal vector to $\po$ and $dS$ is the surface measure on $\po$ (see the proof of Step 1 on page 5 in \cite{Re}).
We compute each term. Firstly, as for $K_2$, we have
\[K_{11} = {1 \over 2} \intom (\pxn a) |\nabla V|^2 + o(1) = {1 \over 2} ka_1 \partial_{\nu}a(\xi_0) + o(1).\]
On the other hand, (2.10) of \cite{Re} gives
\[\intpom \left|\nabla PU_i\right|^2 dS = O\({\delta_i^{n-2} \over \ep^{n-1}}\)\]
and by mimicking the proof of \cite[Lemma 7.2]{MP} or (2.12) in \cite{Re}, one can prove that
\[\intpom |\nabla \phi|^2 dS = o(1).\]
Thus
\begin{align*}
K_{12} &= - {1 \over 2} \intpom a|\nabla V|^2 \nu_n dS + o(1) = - {1 \over 2} \intpom a|\nabla PU_1|^2 \nu_n dS + o(1) \\
&= - \intom a U_1^p (\pxn PU_1) + \left\{\intom (\nabla a \cdot \nabla PU_1) \pxn PU_1 - {1 \over 2} \intom (\pxn a) |\nabla PU_1|^2 \right\} + o(1)\\
&= - p \intom a U_1^{p-1}\psi_1^n PU_1 + \left\{\intom (\nabla a \cdot \nabla PU_1) \pxn PU_1 + \intom (\pxn a) U_1^p PU_1 - {1 \over 2} \intom (\pxn a) |\nabla PU_1|^2 \right\} + o(1)
\end{align*}
(see the proof of Step 2 on page 5 in \cite{Re}). However, we have
\begin{equation}\label{C^1_4}
p \intom a U_1^{p-1}\psi_1^n PU_1 = \({n+2 \over 2n}\) a_1 \partial_{\nu} a(\xi_0) - {1 \over 2} a(\xi_0) a_2 \pt \({d_1 \over 2t}\)^{n-2} + o(1)
\end{equation}
and
\begin{equation}\label{C^1_5}
\intom (\pxn a) |\nabla PU_1|^2,\ \intom (\pxn a) U_1^p PU_1,\ n \intom (\nabla a \cdot \nabla PU_1) \pxn PU_1 = a_1 \partial_{\nu}a(\xi_0) + o(1)
\end{equation}
whose detailed proofs are given below. As a result, we obtain
\[K_{12} = {1 \over 2} a(\xi_0) a_2 \pt \({d_1 \over 2t}\)^{n-2} + o(1)\]
where $a_2$ is given in \eqref{a_2}.

\begin{proof}[Proof of \eqref{C^1_4}]
We write
\begin{equation}\label{C^1_41}
p \intom a U_1^{p-1}\psi_1^n PU_1 = p \intom a U_1^p \psi_1^n + p \intom a U_1^{p-1}\psi_1^n (PU_1-U_1)
\end{equation}
and we estimate the first term in the right-hand side of \eqref{C^1_41}.
By applying \eqref{int_y_zero}, (a3) (in particular, $\la \nabla a(\xi_k), y \ra = \partial_{\nu}a(\xi_k) \cdot y_n$) and Taylor's theorem, \begin{align*}
&\ p \intom a U_1^p \psi_1^n = p \int_{B(\xi_1, \rho\ep)} aU_1^p\psi_1^n + o(1)\\
&= \left[{(n+2)\alpha_n^{p+1} \over \delta_1}\right] \cdot \left[- a(\xi_k) \int_{B(0, \delta^{-1}\rho\ep)^c} {y_n \over (1+|y|^2)^{n+1}}\ dy\right.\\
&\hspace{130pt} \left. + \int_{B(0, \delta_1^{-1}\rho\ep)} \left\{a(\delta_1y+ \delta_1s_1\nu(\xi_0) + \xi_k) - a(\xi_k)\right\} {y_n \over (1+|y|^2)^{n+1}}\ dy \right] + o(1) \\
&= \partial_{\nu}a(\xi_k) \cdot (n+2)\alpha_n^{p+1}\int_{B(0, \delta_1^{-1}\rho\ep)}{y_n^2 \over (1+|y|^2)^{n+1}}\ dy + o(1) = \({n+2 \over 2n}\)\partial_{\nu}a(\xi_0) a_1 + o(1).
\end{align*}

To estimate the second term in the right-hand side of \eqref{C^1_41}, we need
\begin{equation}\label{pxn_H}
\left|(\partial_{x,n} H)(\delta_i y + \xi_i,\xi_j) + (n-2) {(\delta_iy + \xi_i - \xi_j^*)_n \over |\delta_iy + \xi_i - \xi_j^*|^n} \right| = O\({1 \over \ep^{n-2}}\) \quad \text{for } |y| \le \delta_i^{-1}\rho\ep
\end{equation}
where $\nabla H(x,\xi) = (\nabla_x H(x,\xi), \nabla_{\xi} H(x,\xi)) = (\partial_{x,1} H(x,\xi), \cdots, \partial_{x,n} H(x,\xi), \partial_{\xi,1} H(x,\xi), \cdots, \partial_{\xi,n} H(x,\xi))$ and $i, j = 1, \cdots, k$.
  Now, by Lemma \ref{lemma_pre_2}, \ref{lemma_pre_1} and \ref{lemma_pre_01}, \eqref{pxn_H} and the mean value theorem,
\begin{align*}
&~\intom a \(\partial_{(\xi_1)_n} U_1^p\)(PU_1-U_1) = \int_{B(\xi_1, \rho \ep)} a \(\partial_{(\xi_1)_n} U_1^p\) \cdot \alpha_n \delta_1^{n-2 \over 2} H(\cdot, \xi_1) + o(1)\\
&= \int_{B(\xi_1, \rho \ep)} a U_1^p \cdot \alpha_n \delta_1^{n-2 \over 2} \partial_{(\xi_1)_n} (H(\cdot, \xi_1)) + {\partial_{x_n}}\left.\(\int_{B(x, \rho \ep)} a U_1^p \cdot \alpha_n \delta_1^{n-2 \over 2} H(\cdot, \xi_1)\)\right|_{x = \xi_1}\\
&\ \hspace{125pt} - \alpha_n^{p+1} \partial_{(\xi_1)_n}\(\int_{B(0,\delta_1^{-1}\rho\ep)} a(\delta_1 y + \xi_1) {\delta_1^{n-2} \over (1+|y|^2)^{n+2 \over 2}} H(\delta_1 y + \xi_1, \xi_1)\ dy\) + o(1)\\
&= -\alpha_n^{p+1} \int_{B(0,\delta_1^{-1}\rho\ep)} a(\delta_1y + \xi_1) {\delta_1^{n-2} \over (1+|y|^2)^{n+2 \over 2}}(\partial_{x,n}H)(\delta_1y+\xi_1,\xi_1)\ dy + o(1)\\
&= \alpha_n^{p+1}(n-2) \int_{B(0,\delta_1^{-1}\rho\ep)} a(\delta_1y + \xi_1) {\delta_1^{n-2} \over (1+|y|^2)^{n+2 \over 2}}{(2\ep t \nu(\xi_0) + \delta_1 (y + 2s_1\nu(\xi_0)))_n \over |2\ep t \nu(\xi_0) + \delta_1 (y + 2s_1\nu(\xi_0))|^n}\ dy + o(1)\\
& = - {1 \over 2} a(\xi_0)a_2 \pt \({d_1 \over 2t}\)^{n-2} + o(1).
\end{align*}
Hence \eqref{C^1_4} is proved.
\end{proof}

\begin{proof}[Derivation of \eqref{C^1_5}]
By the argument in Section \ref{sec_expansion_1}, we immediately get
\[\intom (\pxn a) |\nabla PU_1|^2,\ \intom (\pxn a) U_1^p PU_1 = a_1 \partial_{\nu}a(\xi_0) + o(1).\]
 On the other hand, by Lemma \ref{lemma_pre_3.75},
\[n \intom (\nabla a \cdot \nabla PU_1) \pxn PU_1 = n \intom (\nabla a \cdot \nabla U_1) \pxn U_1 + o(1).\]
Since (a3) implies $\pxn a(\xi_0) = \partial_{\nu} a(\xi_0)$ and
\[n \intom (\partial_{x_i} a) \cdot (\partial_{x_i} U_1) \cdot (\pxn U_1) = \delta_{in} \cdot \pxn a(\xi_0) \cdot \alpha_n^2(n-2)^2 \int_{\rn} {|y|^2 \over (1+|y|^2)^n} + o(1) =  \delta_{in} \cdot \partial_{\nu} a(\xi_0) a_1+ o(1)\]
for $i = 1, \cdots, n$, \eqref{C^1_5} follows.
\end{proof}

\medskip
In conclusion,
\begin{align*}
J_{\ep}'(u) (\pt u) = \({1 \over 2} - {1 \over p+1}\) ka_1 \partial_{\nu}a(\xi_0) \cdot \ep + {1 \over 2} a(\xi_0) a_2 \cdot \pt \({d_1 \over 2t}\)^{n-2} \ep + o(\ep)
\end{align*}
as desired.
\end{proof}
Consequently, \eqref{s5} for $s = t$ is valid and the proof of Proposition \ref{prop_energy_est} is finished.

\appendix
\section{}
In this appendix, we study functions $PU_{\delta, \xi}$ and $P\psi^j_{\delta, \xi}$ ($j = 0, n$) defined through \eqref{instanton}, \eqref{pdxn}, \eqref{pdxj} and \eqref{proj}.

\subsection{Comparison between $\udx$ and $P\udx$}\label{appendix_W_PW}
Denote by $G(x,y)$ the Green function associated to $-\Delta$ with Dirichlet boundary condition and $H(x,y)$ its regular part: Namely, \[\left\{{\setlength\arraycolsep{2pt}\begin{array}{rll}
-\Delta_x G(x,y) &= \delta_y (x) &\quad \text{for } x \in \Omega,\\
G(x,y) &= 0 & \quad \text{for } x \in \po,
\end{array}} \right.\]
and
\[G(x,y) = \gamma_n \( \frac{1}{|x-y|^{n-2}} - H(x,y) \) \quad \text{ where} \quad \gamma_n = \frac{1}{(n-2)|S^{n-1}|}.\]

\medskip
Since $\Omega$ is smooth, we can choose small $d_0 > 0$ such that, for every $x \in \Omega$ with
$d(x, \po) \le d_0$, there is a unique point $x_{\nu} \in \po$ satisfying $d(x, \po) = |x - x_{\nu}|$.
For such $x \in \Omega$, we define $x^* = 2x_{\nu} - x$ the reflection point of $x$ with respect to $\po$.

\medskip
The following two lemmas are proved in \cite[Appendix A]{ACP} under the assumption that $\Omega$ is of class $C^2$.
\begin{lemma}\label{lemma_pre_1}
There exist a constant $C > 0$ such that
\[\left| H(x, \xi) - {1 \over |x - \xi^*|^{n-2}} \right| \le {C d(\xi, \po) \over |x -\xi^*|^{n-2}}, \quad \quad
\left| \nabla_{\xi} \(H(x, \xi) - {1 \over |x - \xi^*|^{n-2}}\) \right| \le {C \over |x -\xi^*|^{n-2}}\]
and
\[0 \le H(x, \xi) \le {C \over |x - \xi^*|^{n-2}}, \quad |\nabla_\xi H(x, \xi)| \le {C \over |x - \xi^*|^{n-1}}.\]
for any $x \in \Omega$ and $\xi \in \{y \in \Omega : d(y, \po) \le d_0\}$. In particular, we obtain
\[H(x, \xi) \le {C \over |x - \xi|^{n-2}} \quad \text{and} \quad |\nabla_\xi H(x, \xi)| \le {C \over |x - \xi|^{n-1}} \quad \text{for any } x, \xi \in \Omega \]
by taking $C > 0$ larger if necessary.
\end{lemma}

\begin{lemma}\label{lemma_pre_2}
If $\xi \in \{y \in \Omega : d(y, \po) \le d_0\}$, then there exists a constant $C>0$ such that
\begin{equation}\label{pre_2_1}
0 \le \udx(x) - P\udx(x) \le \alpha_n \delta^{n-2 \over 2}H(x,\xi) \le {C \delta^{n-2 \over 2}
\over |x - \xi^*|^{n-2}} \quad \text{for all } x \in \Omega.
\end{equation}
Moreover, it holds true that
\[ P\udx(x) = \udx(x) - \alpha_n \delta^{n-2 \over 2}H(x,\xi) + O\({\delta^{n+2 \over 2} \over d(\xi, \po)^n}\), \quad x \in \Omega.\]
\end{lemma}

\medskip
From the previous lemmas, we can show that
\begin{lemma}\label{lemma_pre_3.5}
Denote $PU_i = PU_{\delta_i, \xi_i}.$   Then
\[\|U_i - PU_i\|_{L^q(\Omega)} = o(1) \quad \text{if } q \in \({n \over n-2}, {2n \over n-3}\)\ \hbox{if}\ n\ge4\ \hbox{or}\ q \in \({n \over n-2},+\infty\)\ \hbox{if}\ n=3.\]
\end{lemma}
\begin{proof}
By \eqref{pre_2_1} and \eqref{conc_para}, we have
\begin{align*}
\|U_i - PU_i\|_{L^q(\Omega)}^q &\le \intom {\delta_i^{(n-2)q \over 2} \over |x - \xi_i^*|^{(n-2)q}}
= \delta_i^{n - {(n-2)q \over 2}} \int_{\Omega - \xi_i \over \delta_i} {dy \over |y + 2\((\ep/\delta_i)t + s_i\)\nu(\xi_0)|^{(n-2)q}}\\
& \le C \delta_i^{n - {(n-2)q \over 2}} \int_{\ep \delta_i^{-1}}^{C\delta_i^{-1}} {s^{n-1} \over s^{(n-2)q}} ds \le C \delta_i^{(n-2)q \over 2} \ep^{n - (n-2)q} \le C \ep^{(n-1)q \over 2} \cdot \ep^{n - (n-2)q}\\
&= O\(\ep^{n - {(n-3)q \over 2}}\) = o(1)
\end{align*}
for some $C>0$.
\end{proof}

\medskip
In addition, we can estimate the $H^1(\Omega)$-norm of $U_i - PU_i$ as follows.
\begin{lemma}\label{lemma_pre_3.75}
It holds true that
\[\|U_i - PU_i\|_{H^1(\Omega)} = O(\sqrt{\ep}). \]
\end{lemma}
\begin{proof}
From the definition \eqref{instanton} of $U_i$ and the fact $\alpha_n^{p-1} = n(n-2)$, we get
\begin{align*}
\|U_i - PU_i\|_{H^1(\Omega)}^2 &= \(\intom |\nabla PU_i|^2 - 2 \intom \nabla PU_i \cdot \nabla U_i\) + \intom |\nabla U_i|^2\\
&= \(\intom U_i^{n+2 \over n-2} PU_i - 2 \intom U_i^{n+2 \over n-2} PU_i\) + \alpha_n^2(n-2)^2 \delta_i^{n-2} \intom {|x-\xi_i|^2 \over (\delta_i^2 + |x-\xi_i|^2)^n}\\
&= \(- \alpha_n^{p+1} \int_{\rn} {1 \over (1+|y|^2)^n} + O(\ep)\) + \(\alpha_n^2(n-2)^2 \int_{\rn} {|y|^2 \over (1+|y|^2)^n} + O(\ep)\)\\
&= O(\ep).
\end{align*}
\end{proof}

\subsection{Estimates of $\psi_i^j$'s}
First, we want  to establish a result similar to the ones proved in   Lemma \ref{lemma_pre_2} and Lemma \ref{lemma_pre_3.5}.
\begin{lemma}\label{lemma_pre_5}
For any $i = 1, \cdots, k$, we have
\begin{align}
P\psi_i^0 &= \psi_i^0 - \alpha_n \({n-2 \over 2}\) \delta_i^{n-4 \over 2} H(\cdot, \xi_i) + O\({\delta_i^{n \over 2} \over \ep^n}\) \quad \text{in } \Omega \label{pre_5_1}
\intertext{and}
P\psi_i^n &= \psi_i^n - \alpha_n \delta_i^{n-2 \over 2} (\partial_{\xi,n}H)(\cdot, \xi_i) + O\({\delta_i^{n+2 \over 2} \over \ep^{n+1}}\) \quad \text{in } \Omega \label{pre_5_2}
\end{align}
where $(\partial_{\xi,n}H)(x,\xi)$ is the $n$-th component of $\nabla_{\xi} H(x,\xi)$.
Moreover,
\begin{equation}\label{pre_5_3}
\big\|\delta_i\big(P \psi_i^j - \psi_i^j\big)\big\|_{L^{2n \over n-2}(\Omega)} = O\(\ep^{n \over n-2}\)
\end{equation}
for $j = 0,\ n$.
\end{lemma}
\begin{proof}
From the comparison principle, we easily deduce  \eqref{pre_5_1} and \eqref{pre_5_2}.
Arguing exactly as in Lemma  \ref{lemma_pre_3.5} and taking into account Lemma \ref{lemma_pre_1}, we can prove  \eqref{pre_5_3}.
\end{proof}

\medskip
The above lemma enables to estimate the difference between $\pxn PU_i$ and $\pxn U_i$ for $i = 1, \cdots, k$. Let $p = (n+2)/(n-2)$.
\begin{lemma}\label{lemma_pre_8}
For $i = 1, \cdots, k$,
\begin{equation}\label{pre_2_2}
\pxn PU_i(x) = \pxn U_i(x) + \alpha_n \delta_i^{n-2 \over 2} (\partial_{\xi,n}H)(x, \xi_i) + O\({\delta_i^{n-2 \over 2} \over \ep^{n-1}}\).
\end{equation}
\end{lemma}
\begin{proof}
Let $w = \pxn PU_i + P\psi_i^n$ so that it solves $\Delta w = 0$ in $\Omega$ and $w = \pxn PU_i$ on $\po$. Then by the maximum principle, $\|w\|_{L^{\infty}(\Omega)} \le \|\pxn PU_i\|_{L^{\infty}(\po)}$. Recalling $H(x,y) = H(y,x)$ and applying Lemma \ref{lemma_pre_1}, we observe that there is a constant $C>0$ such that
\begin{align*}
|\pxn PU_i (x)| &\le \intom |\pxn G(x,y)| \cdot U_i^p(y) dy = \gamma_n (n-2) \intom \left|{(x-y)_n \over |x-y|^n} - (\partial_{\xi,n}H)(y,x)\right| \cdot U_i^p(y) dy\\
& \le C \intom {1 \over |x - y|^{n-1}}\ U_i^p(y) dy.
\end{align*}
Now we choose $\rho > 0$ sufficiently small so that $B(x, \rho \ep) \cap B(\xi_i, \rho \ep) = \emptyset$ for any $x \in \po$. Then for $x \in \po$,
\[\int_{\Omega \cap B(x, \rho \ep)} {1 \over |x - y|^{n-1}} U_i^p(y) dy \le C\({\delta_i^{n+2 \over 2} \over \ep^{n+2}}\) \int_{B(x, \rho \ep)} {1 \over |x - y|^{n-1}} dy= O\({\delta_i^{n+2 \over 2} \over \ep^{n+1}}\)\]
and
\[\int_{\Omega \setminus B(x, \rho \ep)}  {1 \over |x - y|^{n-1}} U_i^p(y) dy \le C\({1 \over \ep^{n-1}}\) \int_{\rn} {\delta_i^{n-2 \over 2} \over (1+|z|^2)^{n+2 \over 2}}\ dz = O\({\delta_i^{n-2 \over 2} \over \ep^{n-1}}\).\]
Therefore we deduce
\[\|\pxn PU_i\|_{L^{\infty}(\Omega)} = O\({\delta_i^{n-2 \over 2} \over \ep^{n-1}}\).\]
Consequently, by \eqref{pre_5_2}, we obtain
\[\pxn PU_i (x) = - P\psi_i^n(x) + O\({\delta_i^{n-2 \over 2} \over \ep^{n-1}}\) = \pxn U_i(x) + \alpha_n \delta_i^{n-2 \over 2} (\partial_{\xi,n}H)(x, \xi_i) + O\({\delta_i^{n-2 \over 2} \over \ep^{n-1}}\).\]
Hence \eqref{pre_2_2} holds.
\end{proof}

\medskip
The next lemma is crucial for the proof of Proposition \ref{prop_linear}.
\begin{lemma}\label{lemma_pre_4}
For $i, \ l = 1, \cdots, k$, $i \le l$ and $j, \ m = 0, \ n$, it holds that
\[\la P\psi_i^j, P\psi_l^m \ra = \left\{ \begin{array}{ll}
a(\xi_0) c_j \dfrac{1}{\delta_i^2} + o\(\dfrac{1}{\delta_i^2}\) &\text{if } i = l \text{ and } j = m, \\
o\(\dfrac{1}{\delta_i^2}\) &\text{otherwise,}
\end{array}\right.\]
where $c_0$ and $c_n$ are positive constants.
\end{lemma}
\begin{proof}
By \eqref{eq_of_psi} we get
\[\la P\psi_i^j, P\psi_l^m \ra = p \intom a U_i^{p-1} \psi_i^j \psi_l^m + p \intom a U_i^{p-1} \psi_i^j \big(P \psi_l^m - \psi_l^m) - \intom (\nabla a \cdot \nabla P\psi_i^j) P\psi_l^m =: M_1 + M_2 + M_3.\]
We will estimate $M_1$, $M_2$ and $M_3$ respectively.

\medskip
To estimate $M_1$, note  that $\delta_{i_1} \ll |\xi_{i_2} - \xi_0|$ for any $i_1, \ i_2 = 1, \cdots, k$.
Then arguing as in the proof of \cite[Lemma A.5]{MP}, we get
\[ M_1 = \left\{ \begin{array}{ll}
a(\xi_0) \dfrac{c_j}{\delta_i^2} + o\(\dfrac{1}{\delta_i^2}\) &\text{if } i = l \text{ and } j = m,\\
o\(\dfrac{1}{\delta_i^2}\) &\text{otherwise,}
\end{array}\right.\]
with positive constants $c_0$ and $c_n$.

Let us estimate $M_2$ when  $j = m = n$.
By \eqref{pre_5_2}, assumption (a1) and Lemma \ref{lemma_pre_1}, we deduce that
\[
M_2 = -\alpha_n \delta_l^{n-2 \over 2} p \int_{B(\xi_i, \rho \ep)} a U_i^{p-1} \psi_i^n (\partial_{\xi,n}H)(\cdot, \xi_l) + O\({\ep^{n+1 \over n-2} \over \delta_i^2}\) = o\(\dfrac{1}{\delta_i^2}\),
\]
where $\rho > 0$ is chosen sufficiently small, since
by  \eqref{conc_para}, assumption (a1), Lemma \ref{lemma_pre_1} and \eqref{pre_5_2} we get
\begin{align*}
\left| \delta_l^{n-2 \over 2} \int_{B(\xi_i, \rho \ep)} a U_i^{p-1} \psi_i^n (\partial_{\xi,n}H)(\cdot, \xi_l)\right|
& \le C\delta_l^{n-2 \over 2} \int_{B(\xi_i, \rho \ep)} {\delta_i^{n+2 \over 2}|x-\xi_i| \over (\delta_i^2 + |x-\xi_i|^2)^{n+4 \over 2}} \cdot {1 \over |x - \xi_i^*|^{n-1}} dx\\
& \le C\delta_l^{n-2\over2} \delta_i^{n-4\over2}\int_{B(0, \rho \ep \delta_i^{-1})} {|y| \over (1 + |y|^2)^{n+4 \over 2}}
 {1 \over  \ep  ^{n-1}}  dy\\
& = O\({\ep^{n-1 \over n-2} \over \delta_i^2}\)
\end{align*}
where $\xi_i^*$ is the reflection of $\xi_i$ with respect to $\po$ defined in the previous subsection and $C>0$ is some constant.
The cases when either $j$ or $m$ is 0  can be carried out in a similar way  using \eqref{pre_5_1}.

Finally, $M_3$ is estimated using Lemma \ref{lemma_pre_3}, which yields to  $M_3 = o\(1/\delta_i^2\).$

This concludes the proof.
\end{proof}

\medskip
Finally, we need
\begin{lemma}\label{lemma_pre_6}
For $i = 1, \cdots, k$, and $j = 0,\ n$, there hold
\begin{align*}
\big\|\pr P\psi_i^j\big\| &= \left\{ \begin{array}{ll}
0 \quad &\text{if } r = d_l \ (l = 1, \cdots, k),\ s_l \ (l = 1, \cdots, k-1),\ l \ne i, \\
O\(\delta_l^{-1}\) \quad &\text{if } r = d_i \text{ or } s_i,
\end{array}\right.
\intertext{and}
\big\|\pt P\psi_i^j\big\| &= O\(\ep \delta_l^{-2}\).
\end{align*}
\end{lemma}
\begin{proof}
For $r = d_1, \cdots, d_k, t, s_1, \cdots, s_{k-1}$,
\[- \Delta \big(\pr P\psi_i^j\big) = p \big(\pr U_i^{p-1}\big) \psi_i^j + p U_i^{p-1} \big(\pr \psi_i^j\big) \quad \text{in } \Omega, \quad \pr P\psi_i^j = 0 \quad \text{on } \po.\]
Therefore
\[\big\|\pr P\psi_i^j\big\| \le C \left\{\big\|\big(\pr U_i^{p-1}\big) \psi_i^j\big\|_{L^{2n \over n+2}(\Omega)} + \big\|U_i^{p-1} \big(\pr \psi_i^j\big)\big\|_{L^{2n \over n+2}(\Omega)}\right\} \]
for some $C > 0$. Now estimate the right-hand side.
\end{proof}

\subsection{Application of Young's inequality}\label{appendix_young}
In this subsection, we gather estimations which can be obtained by Young's inequality. We again denote $p = (n+2)/(n-2)$.

\begin{lemma}\label{lemma_pre_3}
Assume that $i, l = 1, \cdots, k$ and $j,\ m = 0,\ n$. Then we have
\begin{equation}\label{pre_3_1}
\intom |\nabla PU_i| PU_l = o(\ep)
\end{equation}
and
\[ \intom |\nabla PU_i| P\psi_l^m = o\({\ep \over \delta_l}\) \quad \text{and} \quad \intom \big|\nabla P\psi_i^j\big| P\psi_l^m = o\({1 \over \delta_i^2}\).\]
\end{lemma}
\begin{proof}
The proof is essentially given in the proof of \cite[Lemma A.2]{ACP}. For the sake of reader's convenience, we reprove \eqref{pre_3_1}.
Observe that Lemma \ref{lemma_pre_1} tells us that
\[|\nabla PU_i(x)| = \left|\intom \nabla_x G(x,y)U_i^p(y)dy\right| \le C \intom {1 \over |x-y|^{n-1}}U_i^p(y)dy\]
for some constant $C > 0$. Hence, by Young's inequality \cite[Theorem 4.2]{LL},
\[\intom |\nabla PU_i(x)| PU_l(x) dx \le C \intom \intom U_l(x) {1 \over |x-y|^{n-1}} U_i^p(y) dy dx
\le C \|U_l\|_{L^q(\Omega)}\|f\|_{L^r(B(0,M))}\big\|U_i^p\big\|_{L^s(\Omega)}
\]
for any $q, r, s \ge 1$ satisfy $1/q + 1/r + 1/s = 2$, where $f(x) = |x|^{1-n}$ and $M$ is the diameter of $\Omega$.

Fixing $\sigma > 0$ small enough, we choose
\[q = {n \over 1 - (n-1)\sigma} > {n \over n-2} ,\quad r = {n \over (n-1)(1+\sigma)},\quad s = 1.\]
Since
\[\|U_l\|_{L^q(\Omega)} = O\(\delta_l^{{n \over q} - {n-2 \over 2}}\) \quad \text{for } q > {n \over n-2}, \quad \big\|U_i^p\big\|_{L^s(\Omega)} = O\(\delta_i^{{n \over s} - {n+2 \over 2}}\) \quad \text{for }s \ge 1\]
and
$\|f\|_{L^r(B(0,M))} = O(1)$ for $r \in [1, n/(n-1))$, it then follows that
\[\|U_l\|_{L^q(\Omega)}\|f\|_{L^r(B(0,M))}\big\|U_i^p\big\|_{L^s(\Omega)} = O\(\delta_1^{n({1 \over q} + {1 \over s} -1)}\) = O\(\delta_1^{1 - (n-1)\sigma}\) = O\(\ep ^{{n-1 \over n-2} \cdot (1 - (n-1)\sigma)}\) = o(\ep),\]
which gives \eqref{pre_3_1}.
\end{proof}

\medskip
\begin{lemma}\label{lemma_pre_7}
For $i = 1, \cdots, k$ and $j = 0,\ n$,
\[\|\nabla PU_i\|_{L^{2n \over n+2}(\Omega)} = o(\ep) \quad \text{and} \quad \big\|\nabla \pt P\psi_i^j\big\|_{L^{2n \over n+2}(\Omega)} = O\(\ep^{1 - \sigma} \delta_i^{-1}\)\quad \text{if} \quad n\ge4 .\]
\end{lemma}
\begin{proof}
We take into account only $\|\nabla PU_i\|_{L^{2n \over n+2}(\Omega)}$. The other thing can be checked similarly.

Denote $\tp = {2n \over n+2}$ and as the proof of the previous lemma, we compute
\begin{align*}
\|\nabla PU_i\|^{\tp}_{L^{\tp}(\Omega)} &\le C \intom |\nabla PU_i(x)|^{\tp} dx
\le C \intom \intom |\nabla PU_i(x)|^{\tp-1}{1 \over |x-y|^{n-1}}U_i^p(y)dydx\\
& \le C \left\||\nabla PU_i|^{\tp-1}\right\|_{L^{\tp \over \tp-1}(\Omega)} \|f\|_{L^r(B(0,M))} \big\|U_i^p\big\|_{L^s(\Omega)}
\end{align*}
where $f(x) = |x|^{1-n}$ and $M$ is the diameter of $\Omega$ again.
Hence, if $n\ge4$ the choice
\[ r = {n \over (n-1)(1+\sigma)} \quad \text{and} \quad s = {2n \over n+4 - 2(n-1)\sigma} > 1\]
for any sufficiently small $\sigma > 0$ gives
\[\|\nabla PU_i\|_{L^{\tp}(\Omega)} \le C \big\|U_i^p\big\|_{L^s(\Omega)} = O\(\delta_i^{1 - (n-1)\sigma}\) = o(\ep).\]
 \end{proof}

\begin{rmk}\label{rmk_pre}
We point out that the assumption $n \ge 4$ is used in a crucial way in the proof of estimate \eqref{C^1_3.5}.
All the   results necessary to  the proof of the main theorem remain true for $n = 3$ except Lemma \ref{lemma_C^1_4}. In particular,
the proofs of  Proposition \ref{prop_cont} and Lemma \ref{lemma_C^1_2},
  can be slightly modified  when for $n = 3$. Indeed, in the proof of Lemma \ref{lemma_pre_7}, we choose $r = 6/5$ and $s = 1$ to get
$\|\nabla PU_i\|_{L^{6/5}(\Omega)} = O\(\delta_i^{1 \over 2}\) = O(\ep).$  This implies $\|R_4\| = O(\ep)$ in the proof of Proposition \ref{prop_cont}, which is sufficient to conclude the validity of the proposition.
Moreover,
\[\left| \intom \nabla a \cdot \nabla \(\delta_iP\psi_i^j\) \phi\right| \le
C \delta_i \big\|U_i^{p-1} \psi_i^j\big\|_{L^1(\Omega)} \cdot \|\phi\|_{L^6(\Omega)} = O\(\delta_i^{1 \over 2}\) \cdot o(\sqrt{\ep}) = o(\ep),\]
so \eqref{s5_2} holds to be true and the conclusion of Lemma \ref{lemma_C^1_2} is true.

However, when $n=3$ the argument of Lemma \ref{lemma_pre_7} only guarantees $\big\|\nabla \pt P\psi_i^j\big\|_{L^{6 \over 5}(\Omega)} = O\(\delta_i^{-{3 \over 2}}\)$
which does not allow to get the estimate \eqref{C^1_3.5}, since
\[|c_{ij}| \cdot \big\|\nabla \pt P\psi_i^j\big\|_{L^{6 \over 5}(\Omega)} \cdot \|\phi\| = o\(\delta_i^{-{1 \over 2}} \ep^2\) \ne o(\ep) \quad \text{for } i \ge 2.\]
\end{rmk}

\subsection{Differentiation under the integral sign}
Here we recall some useful operations from elementary calculus. (See \cite[Appendix C]{Ev}.)
\begin{lemma}\label{lemma_pre_00}
Let $f: \rn \to \mathbb{R}$ be continuous and integrable. Then
\[{d \over dr} \int_{B(x_0, r)} f(x) dx = \int_{\partial B(x_0, r)} f dS\]
for any $x_0 \in \rn$ and $r > 0$.
\end{lemma}

\begin{lemma}\label{lemma_pre_01}
Suppose $\{U(t)\}_{t \in \mathbb{R}}$ is a family of smooth bounded domains in $\rn$ which depends on $t$ smoothly.
Denote $\mathbf{v}$ as the velocity of the moving boundary $\partial U(t)$ and $\nu$ as the inner unit normal vector to $\partial U(t)$. If $f: \rn \to \mathbb{R}$ is smooth, then
\[{d \over dt} \int_{U(t)} f(x) dx = - \int_{\partial U(t)} f \mathbf{v} \cdot \nu dS.\]
\end{lemma}

\providecommand{\bysame}{\leavevmode\hbox to3em{\hrulefill}\thinspace}
\providecommand{\MR}{\relax\ifhmode\unskip\space\fi MR }
\providecommand{\MRhref}[2]{%
  \href{http://www.ams.org/mathscinet-getitem?mr=#1}{#2}
}
\providecommand{\href}[2]{#2}

\end{document}